\documentclass[a4paper]{scrartcl}
\usepackage[utf8]{inputenc}
\usepackage[english]{babel}
\usepackage{xcolor}
\usepackage{graphicx}
\usepackage{hyperref}
\usepackage[T1]{fontenc}
\usepackage{lmodern}
\usepackage{rotating}
\usepackage{tikz}
\usepackage{tikz-cd}
\usepackage[backend=biber,sortlocale=de_DE,url=false,doi=false,isbn=false,giveninits=true,style=alphabetic,maxbibnames=1000,date=year,sorting=nyt]{biblatex}
\bibliography{../../jabref.bib}
\DeclareFieldFormat*{title}{\mkbibemph{#1\isdot}}
\DeclareFieldFormat*{journaltitle}{#1\isdot}
\renewbibmacro{in:}{}
\usepackage{amsmath}
\usepackage{amsfonts}
\usepackage{amssymb}
\usepackage{amsthm}
\usepackage{mathtools}
\usepackage{stmaryrd}
\usepackage{enumerate}
\usepackage{verbatim}
\usepackage{dsfont}
\usepackage{mathabx}
\usepackage{faktor}
\usepackage{CJKutf8}
\DeclareMathOperator\defect{def}
\DeclareMathOperator\Gal{Gal}

\DeclareMathOperator\conv{conv}
\DeclareMathOperator\QB{QB}
\DeclareMathOperator\wt{wt}

\DeclareMathOperator\avg{avg}
\DeclareMathOperator\supp{supp}

\DeclareMathOperator\DBG{DBG}

\DeclareMathOperator\LP{LP}
\DeclareMathOperator\cl{cl}

\DeclareMathOperator\rk{rk}

\DeclareMathOperator\paths{paths}

\DeclareMathOperator\pad{pad}

\DeclareMathOperator\wts{wts}
\def\GL{{\mathrm{GL}}}

\def\dom{{\mathrm{dom}}}

\hypersetup{colorlinks=false, linkbordercolor={white}, hidelinks, pdfauthor={Felix Schremmer},pdftitle={Affine Deligne--Lusztig varieties via the double Bruhat graph II: Iwahori--Hecke algebra}}
\usepackage{csquotes}
\author{Felix Schremmer}\date{\today}
\title{Affine Deligne--Lusztig varieties via the double Bruhat graph II:\\Iwahori--Hecke algebra}
\numberwithin{equation}{section}
\newtheorem{theorem}[equation]{Theorem}
\newtheorem{proposition}[equation]{Proposition}
\newtheorem{lemma}[equation]{Lemma}
\newtheorem{corollary}[equation]{Corollary}
\newtheorem{conjecture}[equation]{Conjecture}
\theoremstyle{definition}
\newtheorem{definition}[equation]{Definition}

\theoremstyle{remark}
\newtheorem{example}[equation]{Example}
\newtheorem{remark}[equation]{Remark}
\def\abs#1{{\left\lvert{#1}\right\rvert}}

\let\oldqedsymbol\qedsymbol
\def\qedaddendum{}
\def\qedsymbol{\oldqedsymbol\qedaddendum}

\def\af{{\mathrm{af}}}
\def\rightqed{\pushQED{\qed}\qedhere\popQED}

\usepackage{xr}
\newcommand{\crossRef}[3][]{\cite[\protect\NoHyper#2\ref{ex-#3}#1\protect\endNoHyper]{Schremmer2023_orbits}}
\begin{document}
\maketitle
\begin{abstract}
We introduce a new language to describe the geometry of affine Deligne--Lusztig varieties in affine flag varieties. This second part of a two paper series uses this new language, i.e.\ the double Bruhat graph, to describe certain structure constants of the Iwahori--Hecke algebra. As an application, we describe nonemptiness and dimension of affine Deligne--Lusztig varieties for most elements of the affine Weyl group and arbitrary $\sigma$-conjugacy classes.
\end{abstract}
\section{Introduction}
In a seminal paper, Deligne--Lusztig \cite{Deligne1976} introduced a class of varieties, which they use to describe many representations of finite groups of Lie type. An analogous construction yields the so-called affine Deligne--Lusztig varieties, which play an important role e.g.\ in the reduction of Shimura varieties \cite{Rapoport2005, He2018}. Continuing the treatment of \cite{Schremmer2023_orbits}, we study affine Deligne--Lusztig varieties in affine flag varieties.

Let $G$ be a reductive group defined over a local field $F$, whose completion of the maximal unramified extension we denote by $\breve F$. Denote the Frobenius of $\breve F/F$ by $\sigma$ and pick a $\sigma$-stable Iwahori subgroup $I\subseteq G(\breve F)$. The affine Deligne--Lusztig variety $X_x(b)$ associated to two elements $x,b\in G(\breve F)$ is the reduced ind-subscheme of the affine flag variety $G(\breve F)/I$ with geometric points
\begin{align*}
X_x(b) = \{g\in G(\breve F)/I\mid g^{-1} b\sigma(g) \in IxI\}.
\end{align*}
Observe that the isomorphism type of $X_x(b)$ only depends on the $\sigma$-conjugacy class
\begin{align*}
[b] = \{g^{-1} b\sigma(g)\mid g\in G(\breve F)\}
\end{align*}
and the Iwahori double coset $IxI\subseteq G(\breve F)$. These Iwahori double cosets are naturally parametrized by the extended affine Weyl group $\widetilde W$ of $G$, and we get
\begin{align*}
G(\breve F) = \bigsqcup_{x\in\widetilde W}I\dot xI.
\end{align*}
Many geometric properties of the double cosets $I\dot x I$ for various $x\in\widetilde W$ can be understood via the corresponding Iwahori--Hecke algebra $\mathcal H = \mathcal H(\widetilde W)$. This algebra and its representation theory received tremendous interest since the discovery of the Satake isomorphism \cite{Satake1963}. There are a few different and mostly equivalent constructions of this algebra in use. For now, we summarize that this is an algebra over a suitable base field or ring with a basis given by formal variables $T_x$ for $x\in\widetilde W$. The element $T_x\in\mathcal H$ can be thought of as the representation-theoretic analogue of the Iwahori double coset $IxI\subseteq G(\breve F)$. E.g.\ if $x,y\in\widetilde W$, we can write
\begin{align*}
IxI\cdot IyI = \bigcup_{z} IzI
\end{align*}
where the union is taken over all $z\in\widetilde W$ such that the $T_z$-coefficient of $T_x T_y\in\mathcal H$ is non-zero. For a general overview over the structure theory of Iwahori--Hecke algebras and its applications to the geometry of the affine flag variety, we refer to \cite{He2015}.

The set of $\sigma$-conjugacy classes $B(G) = \{[b]\mid b\in G(\breve F)\}$ is the second main object of interest in the definition of affine Deligne--Lusztig varieties. It is a celebrated result of Kottwitz \cite{Kottwitz1985, Kottwitz1997} that each $\sigma$-conjugacy class $[b]$ is uniquely determined by two invariants, known as its Newton point and its Kottwitz point. From He \cite[Theorem~3.7]{He2014}, we get a parametrization of $B(G)$ using the extended affine Weyl group $\widetilde W$. For each $x\in\widetilde W$, consider its $\sigma$-conjugacy class in $\widetilde W$, denoted
\begin{align*}
\mathcal O = \{y^{-1} x\sigma(y)\mid y\in\widetilde W\}.
\end{align*}
Two elements that are $\sigma$-conjugate in $\widetilde W$ will also be $\sigma$-conjugate in $G(\breve F)$, but the converse does not hold true in general. We obtain a surjective but not injective map
\begin{align*}
\{\text{$\sigma$-conjugacy classes }\mathcal O\subseteq \widetilde W\}\rightarrow B(G),
\end{align*}
sending $\mathcal O$ to $[\dot x]\in B(G)$ for any $x\in\mathcal O$.

The analogous construction in the Iwahori--Hecke algebra is the formation of a $\sigma$\mbox{-}twis\-ted cocenter, i.e.\ the quotient of $\mathcal H$ by the submodule $[\mathcal H,\mathcal H]_\sigma$ generated by
\begin{align*}
[h,h']_\sigma = h h' - h' \sigma(h),\qquad h,h'\in\mathcal H.
\end{align*}
An important result of He--Nie \cite[Theorem~C]{He2014b} gives a full description of this cocenter. For each $\sigma$-conjugacy class $\mathcal O\subseteq \widetilde W$ and any two elements of minimal length $x_1, x_2\in \mathcal O$, they prove that the images of $T_{x_1}$ and $T_{x_2}$ in the cocenter of $\mathcal H$ agree. Denoting the common image by $T_{\mathcal O}$, they prove moreover that these $T_{\mathcal O}$ form a basis of the cocenter, parametrized by all $\sigma$-conjugacy classes $\mathcal O\subseteq \widetilde W$.

With these preferred bases $\{T_x\}$ of $\mathcal H$ and $\{T_{\mathcal O}\}$ of the quotient, we obtain structure constants expressing the image of each $T_x$ in the cocenter as linear combination of the $T_{\mathcal O}$'s. These are known as class polynomials, so we write
\begin{align*}
T_x \equiv \sum_{\substack{\mathcal O\subseteq\widetilde W\\\sigma\text{-conj.\ class}}}f_{x,\mathcal O}T_{\mathcal O}\pmod{[\mathcal H,\mathcal H]_\sigma}.
\end{align*}
These representation-theoretic structure constants are often hard to determine. However, they are very useful for studying affine Deligne--Lusztig varieties, especially the following main three questions:
\begin{enumerate}[(Q1)]
\item When is $X_x(b)$ empty? Equivalently, when is the Newton stratum empty?
\item If $X_x(b)\neq\emptyset$, what is its dimension?
\item How many top dimensional irreducible components, up to the action of the $\sigma$-centralizer of $b$, does $X_x(b)$ have?
\end{enumerate}
It is an important result of He that these main questions can be fully answered in terms of the class polynomials, cf.\ \cite[Theorem~6.1]{He2014} and \cite[Theorem~2.19]{He2015}. The class polynomials can moreover be used to count rational points of Newton strata, cf.\ \cite[Proposition~3.7]{He2022}.

In the previous article \cite{Schremmer2023_orbits}, we showed that the same main questions can also be answered, in some cases, using the combinatorial notion of a double Bruhat graph. This is an explicitly described finite graph, introduced by Naito--Watanabe \cite[Section~5.1]{Naito2017} in order to describe periodic $R$-polynomials. Following a result of Görtz--Haines--Kottwitz--Reumann \cite[Section~6]{Goertz2006} comparing affine Deligne--Lusztig varieties with certain intersections in the affine flag variety, we showed that the double Bruhat graph appears naturally as a way to encode certain subvarieties of the affine flag variety.

Write $x = wt^\mu\in\widetilde W,v\in W$, and assume that a regularity condition of the form
\begin{align*}
\forall\alpha\in\Phi^+:~\langle v^{-1} \mu,\alpha\rangle\gg \langle \mu^{\dom}-\nu(b),2\rho\rangle
\end{align*}
is satisfied. Assume moreover that the group $G$ is split over $F$. Then \crossRef{Corollary~}{cor:superregularADLV} shows that the questions of nonemptiness, dimension and top dimensional irreducible components are determined by the set of paths from $v$ to $wv$ in the double Bruhat graph that are increasing with respect to some fixed reflection order $\prec$ and of weight $\mu^{\dom} - \nu(b)$. Our first main result states that this set of paths determines the full class polynomial, and that the assumption of a split group can be removed. 
\begin{theorem}[{Cf.\ Theorem~\ref{thm:classPolynomialViaDbg}}]\label{thm:introClassPolynomials}
Assume that the group $G$ is quasi-split. Let $x = w\varepsilon^\mu\in\widetilde W,v\in W$ and $\mathcal O\subseteq \widetilde W$ such that a regularity condition of the form
\begin{align*}
\forall\alpha\in\Phi^+:~\langle v^{-1} \mu,\alpha\rangle\gg \langle \mu^{\dom}-\nu(\mathcal O),2\rho\rangle
\end{align*}
is satisfied. Then the class polynomial $f_{x,\mathcal O}$ can be expressed in terms of paths in the double Bruhat graph from $v$ to $\sigma(wv)$ that are increasing with respect to some fixed reflection order. For a suitable parametrization of the Iwahori--Hecke algebra as an algebra over the polynomial ring $\mathbb Z[Q]$ (Definition~\ref{def:IHAlgebra}), the class polynomial is explicitly given by
\begin{align*}
f_{x,\mathcal O} = \sum_{p} Q^{\ell(p)},
\end{align*}
where the sum is taken over all paths $p$ in the double Bruhat graph from $v$ to $\sigma(wv)$ that are increasing with respect to some fixed reflection order and  such that $\nu(\mathcal O)$ is the $\sigma$-average of $v^{-1}\mu-\wt(p)$.
\end{theorem}
The assumption of a quasi-split group can be removed following \cite[Section~2]{Goertz2015}, though it requires more cumbersome notation to write down statements in full generality, cf.\ \cite[Section~4.2]{Schremmer2022_newton}.

We will prove Theorem~\ref{thm:introClassPolynomials} as a consequence of the following more fundamental result, computing the structure constants of the multiplication of our standard basis vectors in $\mathcal H$.
\begin{theorem}[{Cf.\ Theorem~\ref{thm:IHProduct}}]\label{thm:introMultiplication}
Let $x=w_x\varepsilon^{\mu_x},z=w_z\varepsilon^{\mu_z}\in \widetilde W$ and $v_z\in W$ satisfying a regularity condition of the form
\begin{align*}
\forall \alpha\in \Phi^+:~\langle v_z^{-1}\mu_z,\alpha\rangle \gg \ell(x).
\end{align*}Define polynomials $\varphi_{x,z,y}$ via
\begin{align*}
T_x T_z = \sum_{y\in \widetilde W} \varphi_{x,z,y}T_{y}\in \mathcal H(\widetilde W).
\end{align*}
Pick an element $y=w_y\varepsilon^{\mu_y}\in \widetilde W$ and $v_x\in W$ such that a regularity condition of the form
\begin{align*}
\forall\alpha\in \Phi^+:~\langle v_x^{-1}\mu_x,\alpha\rangle \gg \ell(x)+\ell(z)-\ell(y)
\end{align*}
is satisfied. Then we can describe the structure constant $\varphi_{x,z,y}$ in terms of paths in the double Bruhat graph. Explicitly, we have $\varphi_{x,z,y}=0$ unless $w_y = (w_x v_x)^{-1} v_y$. In this case, we have
\begin{align*}
\varphi_{x,z,y} = \sum_{p} Q^{\ell(p)},
\end{align*}
where the sum is taken over all paths in the double Bruhat graph from $v_x$ to $w_z v_z$ that are increasing with respect to some reflection order and of weight
\begin{align*}
\wt(p) = v_x^{-1}\mu_x + v_z^{-1}\mu_z - (w_z v_z)^{-1} \mu_y.
\end{align*}
\end{theorem}
Theorem~\ref{thm:IHProduct} below actually proves a stronger statement, requiring only a weaker regularity condition of the form
\begin{align*}
\forall \alpha\in \Phi^+:~\langle v^{-1}\mu_x,\alpha\rangle\gg \ell(x) - \ell(y^{-1}z).
\end{align*}
The resulting description of $\varphi_{x,z,y}$ is more involved however, replacing the single path $p$ by pairs of bounded paths in the double Bruhat graph. Theorem~\ref{thm:introMultiplication} as stated here is sufficient to derive Theorem~\ref{thm:introClassPolynomials}.

So under some very strong regularity conditions, the double Bruhat graph may also be used to understand multiplications of Iwahori double cosets $IxI\cdot IzI$ in $G(\breve F)$. Theorems  \ref{thm:introClassPolynomials} and \ref{thm:introMultiplication} give insight in the generic behaviour of class polynomials and products in the Iwahori--Hecke algebra, solving infinitely many previously intractable questions using a finite combinatorial object. From a practical point of view, this allows us to quickly derive many crucial properties of the weight multisets of the double Bruhat graph by referring to known properties of the Iwahori--Hecke algebra or affine Deligne--Lusztig varieties. Using some of the most powerful tools available to describe affine Deligne--Lusztig varieties and comparing them to the double Bruhat graph, we obtain the following result.
\begin{theorem}[{Cf.\ Theorem~\ref{thm:adlvDimension}}]\label{thm:introADLV}
Let $x = w\varepsilon^\mu\in\widetilde W$ and $v\in W$ satisfying the regularity condition
\begin{align*}
\forall\alpha\in\Phi^+:~\langle v^{-1} \mu,\alpha\rangle\geq 2\rk(G)+14,
\end{align*}
where $\rk(G)$ is the rank of a maximal torus in the group $G$.

Pick an arbitrary $\sigma$-conjugacy class $[b]\in B(G)$. Let $P$ be the set of all paths $p$ in the double Bruhat graph from $v$ to $\sigma(wv)$ that are increasing with respect to some fixed reflection order such that the \emph{$\lambda$-invariant} of $[b]$ (cf.\ \cite[Section~2]{Hamacher2018}) satisfies
\begin{align*}
\lambda(b) = v^{-1}\mu - \wt(p).
\end{align*}

Then $P\neq\emptyset$ if and only if $X_x(b)\neq\emptyset$. If $p$ is a path of maximal length in $P$, then
\begin{align*}
\dim X_x(b) = \frac 12\left(\ell(x)+\ell(p)-\langle \nu(b),2\rho\rangle-\defect(b)\right).
\end{align*}
We give a similar description in terms of the dominant Newton points of $[b]$ rather than the $\lambda$-invariant.
\end{theorem}
Theorem~\ref{thm:introADLV} gives full answers to the questions (Q1) and (Q2) for arbitrary $[b]\in B(G)$ as long as the element $x\in\widetilde W$ satisfies a somewhat mild regularity condition (being linear in the rank of $G$).


The proofs given in this article are mostly combinatorial in nature, and largely independent of its predecessor article \cite{Schremmer2023_orbits}. We will rely only on some basic facts on the double Bruhat graph established in \crossRef{Section~}{sec:DBG}. The best known ways to compute the structure constants of Theorem~\ref{thm:introMultiplication} and the class polynomials $f_{x,\mathcal O}$ are given by certain recursive relations involving simple affine reflections in the extended affine Weyl group. Similarly, the Deligne--Lusztig reduction method of Görtz--He \cite[Section~2.5]{Deligne1976, Goertz2010b} provides such a recursive method to describe many geometric properties of affine Deligne--Lusztig varieties, in particular the ones studied in this paper series. On the double Bruhat side, these are mirrored by the construction of certain bijections between paths due to Naito--Watanabe \cite[Section~3.3]{Naito2017}. We recall these bijections and derive the corresponding properties of the weight multisets in Section~\ref{sec:DBG}. We study the consequences for the Iwahori--Hecke algebra in Section~\ref{sec:ihalgebra}, and the resulting properties of affine Deligne--Lusztig varieties in Section~\ref{sec:ADLV}.

In Section~\ref{sec:outlook}, we finish this series of two papers by listing a number of further-reaching conjectures, predicting a relationship between the geometry of affine Deligne--Lusztig varieties and paths in the double Bruhat graph in various cases. These conjectures are natural generalizations of our results, and withstand an extensive computer search for counterexamples.

Recall that our main goal is to find and prove a description of the geometry of affine Deligne--Lusztig varieties in the affine flag variety that is as concise and precise as the known analogous statements for the affine Grassmannian (as summarized in \crossRef{Theorem~}{thm:hyperspecial}). Our conjectures and partial results towards proving them suggest that the language of the double Bruhat graph is very useful for this task, and might even be the crucial missing piece towards a full description.

We would like to remark that once a conjecture is found that describes the geometry of $X_x(b)$ for arbitrary $x,b$ in terms of the double Bruhat graph, a proof of such a conjecture might simply consist of a straightforward comparison of the Deligne--Lusztig reduction method due to Görtz--He \cite{Deligne1976, Goertz2010b} with the analogous recursive relations of the double Bruhat graph that are discussed in this article.

\section{Acknowledgements}
The author was partially supported by the German Academic Scholarship Foundation, the Marianne-Plehn programme, the DFG Collaborative Research Centre 326 \emph{GAUS} and the Chinese University of Hong Kong. I thank Eva Viehmann, Xuhua He and Quingchao Yu for inspiring discussions, and Eva Viehmann again for her comments on a preliminary version of this article. I am very grateful for the carefuly reading and helpful comments by the annonymous referee.
\section{Notation}
We fix a non-archimedian local field $F$ whose completion of the maximal unramified extension will be denoted $\breve F$. We write $\mathcal O_F$ and $\mathcal O_{\breve F}$ for the respective rings of integers. Let $\varepsilon \in F$ be a uniformizer. The Galois group $\Gamma = \Gal(\breve F/F)$ is generated by the Frobenius $\sigma$.

In the context of Shimura varieties, one would choose $F$ to be a finite extension of the $p$-adic numbers. When studying moduli spaces of shutkas, $F$ would be the field of Laurent series over a finite field.

In any case, we fix a reductive group $G$ over $F$. Via \cite[Section~2]{Goertz2015}, we may reduce questions regarding affine Deligne--Lusztig varieties of $G$ to the case of a quasi-split group. In order to minimize the notational burden, we assume that the group $G$ is quasi-split throughout this paper.

We construct its associated affine root system and affine Weyl group following Haines--Rapoport \cite{Haines2008} and Tits \cite{Tits1979}.

Fix a maximal $\breve F$-split torus $T_{\breve F}\subseteq G_{\breve F}$ and  write $T$ for its centralizer in $G_{\breve F}$, so $T$ is a maximal torus of $G_{\breve F}$. Write $\mathcal A = \mathcal A(G_{\breve F},T_{\breve F})$ for the apartment of the Bruhat-Tits building of $G_{\breve F}$ associated with $T_{\breve F}$. We pick a $\sigma$-invariant alcove $\mathfrak a$ in $\mathcal A$. Its stabilizer is a $\sigma$-invariant Iwahori subgroup $I\subset G(\breve F)$.

Denote the normalizer of $T$ in $G$ by $N_G(T)$. Then the quotient \begin{align*}\widetilde W = N_G(T)(\breve F) / (T(\breve F)\cap I)\end{align*}
is called \emph{extended affine Weyl group}, and $W = N_G(T)(\breve F)/T(\breve F)$ is the \emph{(finite) Weyl group}. The Weyl group $W$ is naturally a quotient of $\widetilde W$. We denote the Frobenius action on $W$ and $\widetilde W$ by $\sigma$ as well.

The affine roots as constructed in \cite[Section~1.6]{Tits1979} are denoted $\Phi_\af$. Each of these roots $a\in \Phi_\af$ defines an affine function $a:\mathcal A\rightarrow\mathbb R$. The vector part of this function is denoted $\cl(a) \in V^\ast$, where $V = X_\ast(S)\otimes\mathbb R = X_\ast(T)_{\Gamma_0}\otimes \mathbb R$. Here, $\Gamma_0 = \Gal(\overline F/\breve F)$ is the absolute Galois group of $\breve F$, i.e.\ the inertia group of $\Gamma = \Gal(\overline F/F)$. The set of \emph{(finite) roots} is\footnote{This is different from the root system that \cite{Tits1979} and \cite{Haines2008} denote by $\Phi$; it coincides with the root system called $\Sigma$ in \cite{Haines2008}.} $\Phi := \cl(\Phi_\af)$.

Each affine root in $\Phi_\af$ divides the standard apartment into two half-spaces, one being the positive and one the negative side. Those affine roots are where our fixed alcove $\mathfrak a$ is on the positive side are called \emph{positive affine roots}. If moreover the alcove $\mathfrak a$ is adjacent to the root hyperplane, it is called \emph{simple affine root}. We denote the sets of simple resp.\ positive affine roots by $\Delta_\af\subseteq \Phi_\af^+\subseteq \Phi_\af$.

Writing $W_\af$ for the extended affine Weyl group of $G$, we get a natural $\sigma$-equivariant short exact sequence (cf.\ \cite[Lemma~14]{Haines2008})
\begin{align*}
1\rightarrow W_\af\rightarrow\widetilde W\rightarrow \pi_1(G)_{\Gamma_0}\rightarrow 1.
\end{align*}
Here, $\pi_1(G) := X_\ast(T)/\mathbb Z\Phi^\vee$ denotes the Borovoi fundamental group.

For each $x\in \widetilde W$, we denote by $\ell(x)\in \mathbb Z_{\geq 0}$ the length of a shortest alcove path from $\mathfrak a$ to $x\mathfrak a$. The elements of length zero are denoted $\Omega$. The above short exact sequence yields an isomorphism of $\Omega$ with $\pi_1(G)_{\Gamma_0}$, realizing $\widetilde W$ as semidirect product $\widetilde W = \Omega\ltimes W_\af$.

Each affine root $a\in \Phi_\af$ defines an affine reflection $r_a$ on $\mathcal A$. The group generated by these reflections is naturally isomorphic to $W_\af$ (cf.\ \cite{Haines2008}), so by abuse of notation, we also write $r_a\in W_\af$ for the corresponding element. We define $S_\af := \{r_a\mid a\in \Delta_\af\}$, called the set of \emph{simple affine reflections}. The pair $(W_\af, S_\af)$ is a Coxeter group with length function $\ell$ as defined above.

We pick a special vertex $\mathfrak x\in \mathcal A$ that is adjacent to $\mathfrak a$. Since we assumed $G$ to be quasi-split, we may and do choose $\mathfrak x$ to be $\sigma$-invariant. We identify $\mathcal A$ with $V$ via $\mathfrak x\mapsto 0$. This allows us to decompose $\Phi_\af = \Phi\times\mathbb Z$, where $a = (\alpha,k)$ corresponds to the function
\begin{align*}
V\rightarrow \mathbb R, v\mapsto \alpha(v)+k.
\end{align*}
From \cite[Proposition~13]{Haines2008}, we moreover get decompositions $\widetilde W = W\ltimes X_\ast(T)_{\Gamma_0}$ and $W_\af = W\ltimes \mathbb Z\Phi^\vee$. Using this decomposition, we write elements $x\in \widetilde W$ as $x = w\varepsilon^\mu$ with $w\in W$ and $\mu\in X_\ast(T)_{\Gamma_0}$. For $a = (\alpha,k)\in \Phi_\af$, we have $r_a = s_\alpha \varepsilon^{k\alpha^\vee}\in W_\af$, where $s_\alpha\in W$ is the reflection associated with $\alpha$. The natural action of $\widetilde W$ on $\Phi_\af$ can be expressed as
\begin{align*}
(w\varepsilon^\mu)(\alpha,k) = (w\alpha,k-\langle\mu,\alpha\rangle).
\end{align*}
We define the \emph{dominant chamber} $C\subseteq V$ to be the Weyl chamber containing our fixed alcove $\mathfrak a$. This gives a Borel subgroup $B\subseteq G$, and corresponding sets of positive/negative/simple roots $\Phi^+, \Phi^-, \Delta\subseteq \Phi$.

By abuse of notation, we denote by $\Phi^+$ also the indicator function of the set of positive roots, i.e.
\begin{align*}
\forall\alpha\in \Phi:~\Phi^+(\alpha)=\begin{cases}1,&\alpha\in \Phi^+,\\
0,&\alpha\in \Phi^-.\end{cases}
\end{align*}
The sets of positive and negative affine roots can be expressed as
\begin{align*}
\Phi_\af^+=&(\Phi^+\times \mathbb Z_{\geq 0})\sqcup (\Phi^-\times \mathbb Z_{\geq 1}) = \{(\alpha,k)\in \Phi_\af\mid k\geq \Phi^+(-\alpha)\},
\\\Phi_\af^- =&-\Phi_\af^+ = \Phi_\af\setminus \Phi_\af^+= \{(\alpha,k)\in \Phi_\af\mid k< \Phi^+(-\alpha)\}.
\end{align*}
One checks that $\Phi_\af^+$ are precisely the affine roots that are sums of simple affine roots.

Decompose $\Phi$ as a direct sum of irreducible root systems, $\Phi = \Phi_1\sqcup\cdots\sqcup \Phi_c$. Each irreducible factor contains a uniquely determined highest root $\theta_i\in \Phi_i^+$. Now the set of simple affine roots is
\begin{align*}
\Delta_\af = \{(\alpha,0)\mid \alpha\in \Delta\}\cup\{(-\theta_i,1)\mid i=1,\dotsc,c\}\subset \Phi_\af^+.
\end{align*}

We call an element $\mu\in X_\ast(T)_{\Gamma_0}\otimes \mathbb Q$ \emph{dominant} if $\langle \mu,\alpha\rangle\geq 0$ for all $\alpha\in \Phi^+$. Similarly, we call it $C$-regular for a real number $C$ if
\begin{align*}
\abs{\langle\mu,\alpha\rangle}\geq C
\end{align*}
for each $\alpha\in \Phi^+$. If $\mu \in X_\ast(T)_{\Gamma_0}$ is dominant, then the Newton point of $\varepsilon^\mu\in \widetilde W$ is given by the $\sigma$-average of $\mu$, defined as
\begin{align*}
\avg_\sigma(\mu) = \frac 1N\sum_{i=1}^N \sigma^i(\mu),
\end{align*}
where $N>0$ is any integer such that the action of $\sigma^N$ on $X_\ast(T)_{\Gamma_0}$ is trivial.

An element $x = w\varepsilon^\mu\in \widetilde W$ is called $C$-regular if $\mu$ is. We write $\LP(x)\subseteq W$ for the set of length positive elements as introduced in \cite[Section~2.2]{Schremmer2022_newton}. If $x$ is $2$-regular, then $\LP(x)$ consists only of one element, namely the uniquely determined $v\in W$ such that $v^{-1}\mu$ is dominant.

For elements $\mu,\mu'$ in $X_\ast(T)_{\Gamma_0}\otimes \mathbb Q$ (resp.\ $X_\ast(T)_{\Gamma_0}$ or $X_\ast(T)_{\Gamma}$), we write $\mu\leq \mu'$ if the difference $\mu'-\mu$ is a $\mathbb Q_{\geq 0}$-linear combination of positive coroots. 
\section{Double Bruhat graph}\label{sec:DBG}
We recall the definition of the double Bruhat graph following Naito--Watanabe \cite[Section~5.1]{Naito2017}. It turns out that the paths we studied in order to understand affine Deligne--Lusztig varieties are a certain subset of the paths studied by Naito--Watanabe in order to study Kazhdan-Lusztig theory, or more precisely periodic $R$-polynomials.
\begin{definition}
Let $\prec$ be a total order on $\Phi^+$, and let moreover $v,w\in W$.
\begin{enumerate}[(a)]
\item The \emph{double Bruhat graph} $\DBG(W)$ is a finite directed graph. Its set of vertices is $W$. For each $w\in W$ and $\alpha\in \Phi^+$, there is an edge $w\xrightarrow{\alpha} ws_\alpha$.
\item A \emph{non-labelled path} $\overline p$ in $\DBG(W)$ is a sequence of adjacent edges
\begin{align*}
\overline p : v = u_1\xrightarrow{\alpha_1} u_2\xrightarrow{\alpha_2} \cdots\xrightarrow{\alpha_\ell} u_{\ell+1} = w.
\end{align*}
We call $\overline p$ a non-labelled path from $v$ to $w$ of length $\ell(\overline p) = \ell$. We say $\overline p$ is \emph{increasing} with respect to $\prec$ if $\alpha_1\prec\cdots\prec\alpha_\ell$. In this case, we moreover say that $\overline p$ \emph{is bounded by }$n\in\mathbb Z$ if $\alpha_\ell = \beta_i$ for some $i\leq n$.
\item A \emph{labelled path} or \emph{path} $p$ in $\DBG(W)$ consists of an unlabelled path
\begin{align*}
\overline p: v = u_1\xrightarrow{\alpha_1}u_2\xrightarrow{\alpha_2}\cdots\xrightarrow{\alpha_\ell} u_{\ell+1}=w
\end{align*}
together with integers $m_1,\dotsc,m_\ell\in\mathbb Z$ subject to the condition
\begin{align*}
m_i\geq \Phi^+(-u_i\alpha_i)=\begin{cases}0,&\ell(u_{i+1})>\ell(u_i),\\
1,&\ell(u_{i+1})<\ell(u_i).\end{cases}
\end{align*}
We write $p$ as
\begin{align*}
p : v = u_1\xrightarrow{(\alpha_1,m_1)}u_2\xrightarrow{(\alpha_2,m_2)}\cdots\xrightarrow{(\alpha_\ell,m_\ell)} u_{\ell+1}=w.
\end{align*}
The \emph{weight} of $p$ is
\begin{align*}
\wt(p) = m_1\alpha_1^\vee+\cdots+m_\ell\alpha_\ell^\vee\in\mathbb Z\Phi^\vee.
\end{align*}
The \emph{length} of $p$ is $\ell(p) = \ell(\overline p) = \ell$. We say that $p$ is \emph{increasing} with respect to $\prec$ if $\overline p$ is. In this case, we say that $p$ is bounded by $n\in\mathbb Z$ if $\overline p$ is.
\item The set of all paths from $v$ to $w$ that are increasing with respect to $\prec$ and bounded by $n\in\mathbb Z$ is denoted $\paths^\prec_{\preceq n}(v\Rightarrow w)$. We also write \begin{align*}\paths^\prec(v\Rightarrow w) = \paths^\prec_{\preceq\#\Phi^+}(v\Rightarrow w).\end{align*}
\item The order $\prec$ is called a \emph{reflection order} if for all roots $\alpha,\beta\in \Phi^+$ with $\alpha+\beta\in \Phi^+$, we have
\begin{align*}
\alpha\prec \alpha+\beta\prec\beta\text{ or }\beta\prec\alpha+\beta\prec \alpha.
\end{align*}
\end{enumerate}
\end{definition}
We will frequently use the immediate properties of these paths as developed in \crossRef{Section~}{sec:DBG}. For this section, our main result describe how these paths behave with respect to certain simple affine reflections. Fix a reflection order
\begin{align*}
\Phi^+ = \{\beta_1\prec\cdots\prec\beta_{\#\Phi^+}\}
\end{align*}
and write
\begin{align*}
\pi_{\succ n} = s_{\beta_{n+1}}\cdots s_{\beta_{\#\Phi^+}}\in W
\end{align*}
as in \crossRef{Definition~}{def:weightMultiset}.
\begin{theorem}\label{thm:dbgRecursions}
Let $u,v\in W$ and $n\in\{0,\dotsc,\#\Phi^+\}$. Pick a simple affine root $a = (\alpha,k)\in \Delta_\af$ such that $(v \pi_{\succ n})^{-1}\alpha\in \Phi^-$.
\begin{enumerate}[(a)]
\item If $u^{-1}\alpha\in \Phi^-$, then there exists an explicitly described bijection of paths
\begin{align*}
\psi: \paths^\prec_{\preceq n}(s_\alpha u\Rightarrow s_\alpha v)\rightarrow \paths^\prec_{\preceq n}(u\Rightarrow v)
\end{align*}
satisfying for each $p\in \paths^\prec_{\preceq n}(s_\alpha u\Rightarrow s_\alpha v)$ the conditions
\begin{align*}
\ell(\psi(p)) = \ell(p),\qquad \wt(\psi(p)) = \wt(p) + k(v^{-1}\alpha^\vee - u^{-1}\alpha^\vee).
\end{align*}
\item If $u^{-1}\alpha\in \Phi^+$, then there exists an explicitly described bijection of paths
\begin{align*}
\varphi: \paths^\prec_{\preceq n}(s_\alpha u\Rightarrow s_\alpha v)\sqcup  \paths^\prec_{\preceq n}(s_\alpha u\Rightarrow v)\rightarrow \paths^\prec_{\preceq n}(u\Rightarrow v)
\end{align*}
satisfying for each $p\in \paths^\prec_{\preceq n}(s_\alpha u\Rightarrow s_\alpha v)$ and $p'\in \paths^\prec_{\preceq n}(s_\alpha u\Rightarrow v)$ the conditions
\begin{align*}
\ell(\varphi(p)) =& \ell(p),\qquad \wt(\varphi(p)) = \wt(p) + k(v^{-1}\alpha^\vee - u^{-1}\alpha^\vee),\\
\ell(\varphi(p')) =& \ell(p')+1,\qquad \wt(\varphi(p')) = \wt(p') - ku^{-1}\alpha^\vee.
\end{align*}
\end{enumerate}
\end{theorem}
The proof of this theorem can essentially be found in Section~3.3 of \cite{Naito2017}, which is a rather involved and technical construction. One may obtain a weaker version of Theorem~\ref{thm:dbgRecursions} by comparing the action of simple affine reflections on semi-infinite orbits with \crossRef{Theorem~}{thm:semiInfiniteOrbitsViaPaths}. While such a weaker result would be sufficient for our geometric applications, we do need the full strength of Theorem~\ref{thm:dbgRecursions} for our conclusions on the Iwahori--Hecke algebra. Moreover, we would like to explain the connection between our paper and \cite{Naito2017}. Let us hence recall some of the notation used by Naito--Watanabe:
\begin{definition}
\begin{enumerate}[(a)]
\item By $\leq_{\frac\infty 2}$, we denote the semi-infinite order on $\widetilde W$ as introduced by Lusztig \cite{Lusztig1980}. It is generated by inequalities of the form
\begin{align*}
w\varepsilon^\mu<_{\frac\infty 2} r_{(\alpha,k)}w\varepsilon^\mu
\end{align*}
where $(\alpha,k)\in \Phi_\af^+$, $w\in W$ and $\mu\in X_\ast(T)_{\Gamma_0}$ satisfy $w^{-1}\alpha\in \Phi^+$.
\item For $w, y\in \widetilde W$, we denote by $P_r^\prec(y, w)$ the set of paths in $\widetilde W$ of the form
\begin{align*}
\Pi:~y = y_1\xrightarrow{(\beta_1, m_1)}y_2\xrightarrow{(\beta_2, m_2)}\cdots\xrightarrow{(\beta_{\ell}, m_{\ell})} y_{\ell+1} = w
\end{align*}
such that the following two conditions are both satisfied:
\begin{itemize}
\item For each $i=1,\dotsc,\ell$, we have $y_{i+1}>_{\frac\infty 2} y_i$. Writing $y_i = w_i\varepsilon^{\mu_i}$, we have
\begin{align*}
y_{i+1} = w_i s_{\beta_i} \varepsilon^{\mu_i + m_i \beta_i^\vee}.
\end{align*}
\item The roots $\beta_i$ are all positive and satisfy $\beta_1\prec\cdots\prec \beta_\ell$.
\end{itemize}
We denote the number of edges in $\Pi$ by $\ell(\Pi):=\ell$.
\end{enumerate}
\end{definition}
These paths $P^\prec_r(\cdot,\cdot)$ occur with exactly the same name in the article of Naito--Watanabe, and are called translation-free paths. They also consider a larger set of paths, where so-called translation edges are allowed, which is however less relevant for our applications.

From the definition of the semi-infinite order, we easily obtain the following relation between the paths in $\widetilde W$ and the paths in the double Bruhat graph. This can be seen as a variant of \cite[Proposition~5.2.1]{Naito2017}.
\begin{lemma}\label{lem:pathTranslation}
Let $y = w_1\varepsilon^{\mu_1}, w  = w_2\varepsilon^{\mu_2}\in \widetilde W$. Then the map
\begin{align*}
\Psi:&P_r^\prec(y, w)\rightarrow \{p\in \paths^\prec(w_1\Rightarrow w_2)\mid \wt(p) = \mu_2-\mu_1\},
\\&\left(\Pi: y = y_0\xrightarrow{(\beta_1, m_1)}y_1\xrightarrow{(\beta_2, m_2)}\cdots\xrightarrow{(\beta_{\ell}, m_{\ell})} y_{\ell+1} = w\right)\\&\mapsto \left(\Phi(\Pi):~w_1 = \cl(y_0)\xrightarrow{(\beta_1, m_1)}\cl(y_1)\xrightarrow{(\beta_2, m_2)}\cdots\xrightarrow{(\beta_{\ell}, m_{\ell})} \cl(y_{\ell+1})\right)
\end{align*}
is bijective and length-preserving (i.e.\ $\ell(\Psi(\Pi)) = \ell(\Pi)$).\rightqed
\end{lemma}
The main results of \cite[Section~3.3]{Naito2017} can be summarized as follows.
\begin{theorem}\label{thm:nwBijections}
Let $y, w\in \widetilde W$ and pick a simple affine reflection $s\in S_\af$ such that $y<_{\frac\infty 2}sy$ and $sw<_{\frac \infty 2}w$.
\begin{enumerate}[(a)]
\item \cite[Proposition~3.3.2]{Naito2017}: There is an explicitly described bijection
\begin{align*}
\psi: P^\prec_r(y, sw)\rightarrow P^\prec_r(sy, w).
\end{align*}
The map $\psi$ preserves the lengths of paths. Its inverse map $\psi' =\psi^{-1}$ is also explicitly described.
\item \cite[Proposition~3.3.1]{Naito2017}: There is an explicitly described bijection
\begin{align*}
\varphi: P^\prec_r(sy, sw)\sqcup P^\prec_r(sy, w)\rightarrow P^\prec_r(y, w).
\end{align*}
For $\Pi \in P^\prec_r(sy, sw)$, we have $\ell(\varphi(\Pi)) = \ell(\Pi)$. For $\Pi\in P^\prec_r(sy, w)$, we have $\ell(\varphi(\Pi)) = \ell(\Pi)+1$. Its inverse map $\varphi' = \varphi^{-1}$ is also explicitly described.\rightqed
\end{enumerate}
\end{theorem}
In view of Lemma~\ref{lem:pathTranslation}, we immediately get the special case of Theorem~\ref{thm:dbgRecursions} for the sets $\paths^\prec(u\Rightarrow v)$, i.e.\ if $n = \#\Phi^+$.
By inspecting the proof and the explicit constructions involved in the proof of Theorem~\ref{thm:nwBijections}, we will obtain the full statement of Theorem~\ref{thm:dbgRecursions}. In order to facilitate this task, we introduce a technique that we call \enquote{path padding}.
\begin{definition}
Let $u,v\in W$ and $0\leq n\leq\#\Phi^+$. Fix positive integers $m_i$ for $i=1,\dotsc,\#\Phi^+$. Then we define the \emph{padding map}
\begin{align*}
\mathrm{pad}_{(m_i)} : \paths^\prec_{\preceq n}(u\Rightarrow v)\rightarrow \paths^\prec(u \Rightarrow v\pi_{\succ n}),
\end{align*}
sending a path $p\in \paths^\prec_{\preceq n}(u\Rightarrow v)$ to the composite path
\begin{align*}
\mathrm{pad}_{(m_i)}(p): u\xRightarrow{p}&
v\xrightarrow{(\beta_{n+1}, m_{n+1})} vs_{\beta_{n+1}}\xrightarrow{(\beta_{n+2}, m_{n+2})}\cdots\\&\cdots \xrightarrow{(\beta_{\#\Phi^+}, m_{\#\Phi^+})}vs_{\beta_{n+1}}\cdots s_{\beta_{\#\Phi^+}} = v\pi_{\succ n}.
\end{align*}
\end{definition}

\begin{lemma}\label{lem:dbgRecursions}
Let $u,v\in W$ and $0\leq n\leq \#\Phi^+$. Pick a simple affine root $a = (\alpha,k)\in \Delta_\af$ such that $(v \pi_{\succ n})^{-1}\alpha\in \Phi^-$.
\begin{enumerate}[(a)]
\item Suppose that $u^{-1}\alpha\in \Phi^-$. For each collection of integers $(m_i\geq 4)_{1\leq i\leq \#\Phi^+}$, there is a unique map \begin{align*}
\tilde \psi :  \paths^\prec_{\preceq n}(s_\alpha u\Rightarrow s_\alpha v)\rightarrow \paths^\prec_{\preceq n}(u\Rightarrow v)
\end{align*} and a collection of integers $(m'_i\geq m_{i}-3)_{1\leq i\leq\#\Phi^+}$ such that the following diagram commutes:
\begin{align*}
\begin{tikzcd}[ampersand replacement=\&,column sep=1em]
\paths^\prec_{\preceq n}(s_\alpha u\Rightarrow s_\alpha v)\ar[d,dotted,"{\tilde\psi}"]\ar[r,"{\mathrm{pad}_{(m_i)}}",hook]\&\paths^\prec(s_\alpha u\Rightarrow s_\alpha v\pi_{\succ n})\ar[r,"{\Psi^{-1}}","{\sim}"']\&\bigsqcup\limits_{\mu\in \mathbb Z\Phi^\vee}P^\prec_r(r_a u, r_a v\pi_{\succ n}\varepsilon^\mu)\ar[d,"{\psi}","\sim"']\\
\paths^\prec_{\preceq n}(u\Rightarrow v)\ar[r,"{\mathrm{pad}_{(m'_i)}}",hook]\&\paths^\prec(u \Rightarrow v\pi_{\succ n})\ar[r,"{\Psi^{-1}}","{\sim}"']\&\bigsqcup\limits_{\mu\in \mathbb Z\Phi^\vee}P^\prec_r(u, v\pi_{\succ n}\varepsilon^\mu).
\end{tikzcd}
\end{align*}
The map $\psi$ on the right comes from Theorem \ref{thm:nwBijections} (a).
The map $\tilde\psi$ has an explicit description independent of the integers $(m_i)$. Moreover, $\tilde\psi$ satisfies the weight and length constraints as required in Theorem \ref{thm:dbgRecursions} (a).

Similarly, there exist integers $(m''_i\geq m_i-3)_{i}$ and a uniquely determined and explicitly described map $\tilde\psi'$ making the following diagram commute:
\begin{align*}
\begin{tikzcd}[ampersand replacement=\&,column sep=1em]
\paths^\prec_{\preceq n}(u\Rightarrow v)
\ar[d,dotted,"{\tilde\psi'}"]\ar[r,"{\mathrm{pad}_{(m_i)}}",hook]
\&\ paths^\prec(u\Rightarrow v\pi_{\succ n})
\ar[r,"{\Psi^{-1}}","{\sim}"']
\& \bigsqcup\limits_{\mu\in \mathbb Z\Phi^\vee}P^\prec_r(u, v\pi_{\succ n}\varepsilon^\mu)
\ar[d,"{\psi'}","\sim"']
\\
\paths^\prec_{\preceq n}(s_\alpha u\Rightarrow s_\alpha v)
\ar[r,"{\mathrm{pad}_{(m''_i)}}",hook]
\&\paths^\prec(s_\alpha u\Rightarrow s_\alpha v\pi_{\succ n})
\ar[r,"{\Psi^{-1}}","{\sim}"']
\&\bigsqcup\limits_{\mu\in \mathbb Z\Phi^\vee}P^\prec_r(r_a u, r_a v\pi_{\succ n}\varepsilon^\mu).
\end{tikzcd}
\end{align*}
\item Suppose that $u^{-1}\alpha\in \Phi^+$. For each collection of integers $(m_i\geq 4)_{1\leq i\leq\#\Phi^+}$, the explicitly described maps
\begin{align*}
\varphi_1 :& \bigsqcup_{\mu\in\mathbb Z\Phi^\vee}P_r^\prec(r_a u,r_a v\pi_{\succ n}\varepsilon^\mu)\rightarrow \bigsqcup_{\mu\in\mathbb Z\Phi^\vee}P_r^\prec(u,v\pi_{\succ n}\varepsilon^\mu),\\
\varphi_2:& \bigsqcup_{\mu\in\mathbb Z\Phi^\vee}P_r^\prec(r_a u, v\pi_{\succ n}\varepsilon^\mu)\rightarrow \bigsqcup_{\mu\in\mathbb Z\Phi^\vee}P_r^\prec(u,v\pi_{\succ n}\varepsilon^\mu),\\
\varphi':&\bigsqcup_{\mu\in\mathbb Z\Phi^\vee} P^\prec_r(u,v\pi_{\succ n}\varepsilon^\mu)\rightarrow \bigsqcup_{\mu\in\mathbb Z\Phi^\vee} P^\prec_r(r_a u,r_a v \pi_{\succ n}\varepsilon^\mu)\sqcup P^\prec_r(r_a u,v \pi_{\succ n}\varepsilon^\mu)
\end{align*}
from Theorem \ref{thm:nwBijections} (b) can be lifted, up to padding and $\Psi^{-1}$ as in (a), to uniquely determined maps
\begin{align*}
\tilde \varphi_1 :& \paths^\prec_{\preceq n}(s_\alpha u\Rightarrow s_\alpha v) \rightarrow \paths^\prec_{\preceq n}(u\Rightarrow v),\\
\tilde \varphi_2 :& \paths^\prec_{\preceq n}(s_\alpha u\Rightarrow v) \rightarrow \paths^\prec_{\preceq n}(u\Rightarrow v),\\
\tilde\varphi':&\paths^\prec_{\preceq n}(u\Rightarrow v)\rightarrow \paths^\prec_{\preceq n}(s_\alpha u\Rightarrow s_\alpha v)\sqcup \paths^\prec_{\preceq n}(s_\alpha u\Rightarrow s_\alpha v).
\end{align*}
All three maps are explicitly described in a way that is independent of the integers $(m_i)$.
The maps $\varphi_1$ and $\varphi_2$ moreover satisfy the desired length and weight compatibility relations from Theorem \ref{thm:dbgRecursions} (b).
\end{enumerate}
\end{lemma}
\begin{proof}
We only explain how to obtain the map $\tilde\psi$ from the map $\psi$, as the other cases are analogous. So pick any path $p\in \paths^\prec_{\preceq n}(s_\alpha u\Rightarrow s_\alpha v)$. Write it as
\begin{align*}
p:~s_\alpha u = w_1\xrightarrow{(\gamma_1, n_1)}w_2\xrightarrow{(\gamma_2, n_2)}\cdots\xrightarrow{(\gamma_{\ell(p)}, n_{\ell(p)})} w_{\ell(p)+1} = s_\alpha v.
\end{align*}
Then
\begin{align*}
\mathrm{pad}_{(m_i)}(p):& s_\alpha  u = w_1\xrightarrow{(\gamma_1, n_1)}\cdots \xrightarrow{(\gamma_{\ell(p)}, n_{\ell(p)})}s_\alpha w_{\ell(p)+1}=s_\alpha v\\&\xrightarrow{(\beta_{n+1}, m_{n+1})}\cdots\xrightarrow{(\beta_{\#\Phi^+}, m_{\#\Phi^+})}s_\alpha v\pi_{\succ n}.
\end{align*}
Define $\gamma_{\ell(p)+i} = \beta_{n+i}$ and $n_{\ell(p)+i} = m_{\ell(p)+i}$ for $i=1,\dotsc,\#\Phi^+-\ell(p)$. Then we can write
\begin{align*}
\mathrm{pad}_{(m_i)}(p):s_\alpha u=w_1\xrightarrow{(\gamma_1, n_1)}\cdots\xrightarrow{(\gamma_{\ell'}, n_{\ell'})}w_{\ell'+1} = v\pi_{\succ n},
\end{align*}
such that $\ell' = \ell(p) + (\#\Phi^+-n)$.

Writing $\mu := \wt(\pad_{(m_i)}(p)) + k\left((v\pi_{\succ n})^{-1}\alpha^\vee - u^{-1}\alpha^\vee\right)$, we may express the path $\Pi :=  \Psi^{-1}(\pad_{(m_i)}(p)) \in P_r^\prec\left(r_a u, r_a v\pi_{\succ\beta_n}\varepsilon^\mu\right)$ as
\begin{align*}\Pi:&~r_a u = w_1\varepsilon^{-kw_1^{-1}\alpha^\vee}\xrightarrow{(\gamma_1, n_1)}w_2\varepsilon^{n_1\gamma_1^\vee-kw_1^{-1}\alpha^\vee}\xrightarrow{(\gamma_2, k_2)}\\&\cdots\xrightarrow{(\gamma_{\ell'}, n_{\ell'})}w_{\ell'+1}\varepsilon^{\wt(\pad_{(m_i)}(p))-kw_1^{-1}\alpha^\vee} = r_av\pi_{\succ n}\varepsilon^\mu.
\end{align*}
We now apply the map $\psi$ as defined in \cite[Section~3.3]{Naito2017}. For this, we need to determine the set
\begin{align*}
D_{r_a}(\Pi) = \{d\in \{1,\dotsc,\ell'\}\mid (\alpha,k) = (w_d^{-1}\gamma_d, n_d)\}.
\end{align*}
Since $m_i\geq 4$ for all $i$, we get
\begin{align*}
D_{r_a}(\Pi) = \{d \mid d\in \{1,\dotsc,\ell(p)\}\text{ and }(\alpha,k) = (w_d^{-1}\gamma_d, n_d)\}\subseteq [1,\ell(p)].
\end{align*}
In particular, the set $D_{r_a}(\Pi)$ depends only on $p$ and not the integers $(m_i)$. 

Naito--Watanabe construct the path $\psi(\Pi)$ as follows: Write $D_{r_a}(\Pi) = \{d_1<\cdots<d_m\}$, which we allow to be the empty set.

For each index $q\in \{1,\dotsc,m\}$, we define $r_q\in \{d_q+2,\dotsc,d_{q+1}\}$ (where $d_{m+1} = \ell'+1$) to be the smallest index such that
\begin{align*}
w_{r_q}^{-1}\alpha\in \Phi^+\text{ and }\gamma_{r_q-1}\prec w_{r_q}^{-1}\alpha\prec \gamma_{r_q}.
\end{align*}
The existence of such an index $r_q$ is proved in \cite[Lemma~2.3.2]{Naito2017}. For $i=1,\dotsc,\#\Phi^+-n$, note that there is no positive root $\beta$ satisfying $\gamma_i\prec\beta\prec\gamma_{i+1}$ (resp.\ $\gamma_{\ell'}\prec\beta$ if $i=\#\Phi^+-n\geq 1$). Hence $r_1,\dotsc,r_m\leq n$ and they only depend on the path $p$, not the integers $(m_i)$.


We introduce the shorthand notation
\begin{align*}
x_h := w_h\varepsilon^{n_1\gamma_1^\vee+\cdots + n_{h-1}\gamma_{h-1}^\vee - kw_1^{-1}\alpha^\vee},
\end{align*}
such that $\Pi$ is of the form $x_1\rightarrow\cdots\rightarrow x_{\ell'+1}$. Then $\psi(\Pi)$ is defined as the composition of $\Pi_0',\dotsc,\Pi_m'$, given by
\begin{align*}
\Pi'_0:&~u=r_a x_1\xrightarrow{(\gamma_1, n_1')}r_a x_2\xrightarrow{(\gamma_2, n_2')}\cdots\xrightarrow{(\gamma_{d_1-1}, n'_{d_1-1})}r_ax_{d_1},\\
\Pi'_q:&~r_a x_{d_q} = x_{d_q+1}\xrightarrow{(\gamma_{d_q+1}, n_{d_q+1})}\cdots \xrightarrow{(\gamma_{r_q-1}, n_{r_q-1})}x_{r_q}\xrightarrow{(w_{r_q}^{-1}\alpha, k)} r_a x_{r_q}\xrightarrow{(\gamma_{r_q}, n'_{r_q})}\cdots\\&\qquad\xrightarrow{(\gamma_{d_{q+1}-1}, n'_{d_{q+1}-1})} r_a x_{d_{q+1}},
\end{align*}
where we write
\begin{align*}
n'_i := n_i + k\langle \alpha^\vee,w_i\gamma_i\rangle,\qquad i=1,\dotsc,\ell'.
\end{align*}
Since $r_1,\dotsc,r_m\leq n$, we may write $\psi(\Pi) = \Psi^{-1}(\pad_{(m'_i)}(p'))$ with
\begin{align*}
m'_i = m_i - k\langle \alpha^\vee, v s_{\beta_{n+1}}\cdots s_{\beta_{i-1}}(\beta_i)\rangle,\qquad i>n.
\end{align*}
The path $p'$ is the composition of the paths $p'_0,\dotsc,p'_m$ defined as
\begin{align*}
p'_0:&~u = s_\alpha w_1\xrightarrow{(\gamma_1, n_1')}\cdots\xrightarrow{(\gamma_{d_1-1}, n'_{d_1-1})} s_\alpha w_{d_1-1}\\
p'_q:&~s_\alpha w_{d_q} = w_{d_q+1}\xrightarrow{(\gamma_{d_q+1}, n_{d_q+1})}\cdots \xrightarrow{(\gamma_{r_q-1}, n_{r_q-1})}w_{r_q}\xrightarrow{(w_{r_q}^{-1}\alpha, k)} s_\alpha w_{r_q}\xrightarrow{(\gamma_{r_q}, n'_{r_q'})}\cdots\\&\qquad\xrightarrow{(\gamma_{d_{q+1}-1}, n'_{d_{q+1}-1})} s_\alpha w_{d_{q+1}}.
\end{align*}
We see that $p'$ as defined above is explicitly described only in terms of $p$ and independently of the $(m_i)$.

To summarize: We chose integers $(m'_i)$ only depending on $(m_i), u,v,n,\prec,a$ with the following property: For each path $p\in \paths^\prec_{\preceq n}(s_\alpha u\Rightarrow s_\alpha v)$, we may write \begin{align*}\psi(\Psi^{-1}\pad_{(m_i)}(p)) = \Psi^{-1}(\pad_{(m'_i)}(p'))\text{ for some path }p'\in \paths^\prec_{\preceq n}(u\Rightarrow v).\end{align*}
It follows that the function $\tilde\psi$ as claimed exists. It is uniquely determined since $\Psi^{-1}$ and $\pad_{(m'_i)}$ are injective. Moreover, we saw that $p':=\tilde\psi(p)$ can be explicitly described depending only on $p$ and not the integers $(m_i)$.

The function $\tilde\psi$ preserves lengths of paths by construction. Using the explicit description, it is possible to verify that it also satisfies the weight constraint stated in Theorem~\ref{thm:dbgRecursions} (a). The interested reader is invited to verify that the constructions of $\psi', \varphi_1, \varphi_2, \varphi'$ of Naito--Watanabe carry through in similar ways.
\end{proof}
With the main lemma proved, we can conclude Theorem~\ref{thm:dbgRecursions} immediately. Indeed, it remains to show that the functions $\tilde\psi$ and $\tilde\varphi :=(\tilde\varphi_1,\tilde\varphi_2)$ from Lemma~\ref{lem:dbgRecursions} are bijective. Since $\psi$ is bijective with $\psi'$ being its inverse, it follows from the categorical definition and a bit of diagram chasing that $\tilde\psi$ is bijective with $\tilde\psi'$ its inverse. Similarly, one concludes that $\tilde\varphi$ is bijective with $\tilde\varphi'$ its inverse. The main result of this section is proved.

\begin{remark}
\begin{enumerate}[(a)]
\item
Theorem~\ref{thm:dbgRecursions} can be conveniently restated using the language of weight multisets from \crossRef{Definition~}{def:weightMultiset}. For $u,v\in W$ and $0\leq n\leq \#\Phi^+$, we write $\wts(u\Rightarrow v\dashrightarrow v\pi_{\succ n})$ for the multiset
\begin{align*}
\{(\wt(p),\ell(p))\mid p\in\paths^\prec_{\preceq n}(u\Rightarrow v)\}_m.
\end{align*}
We proved that this yields a well-defined multiset $\wts(u\Rightarrow v\dashrightarrow v')$ for all $u,v,v'\in W$.

If $a = (\alpha,k)\in\Delta_\af$ is a simple affine root with $(v')^{-1}\alpha\in \Phi^-$ and $u^{-1}\alpha\in \Phi^-$, then
\begin{align*}
\wts(u\Rightarrow v\dashrightarrow v') = \{(\omega + k(v^{-1}\alpha^\vee-u^{-1}\alpha^\vee),e)\mid (\omega,e)\in \wts(s_\alpha u\Rightarrow s_\alpha v\dashrightarrow s_\alpha v')\}_m.
\end{align*}
If $(v')^{-1}\alpha\in \Phi^-$ and $u^{-1}\alpha\in \Phi^+$, then $\wts(u\Rightarrow v\dashrightarrow v')$ is the additive union of the two multisets
\begin{align*}
&\{(\omega + k(v^{-1}\alpha^\vee-u^{-1}\alpha^\vee),e)\mid (\omega,e)\in \wts(s_\alpha u\Rightarrow s_\alpha v\dashrightarrow s_\alpha v')\}_m
\\\cup&\{(\omega - ku^{-1}\alpha^\vee,e)\mid (\omega,e)\in \wts(s_\alpha u\Rightarrow v\dashrightarrow v')\}_m.
\end{align*}
\item The double Bruhat graph can be seen as a generalization of the quantum Bruhat graph, cf.\ \crossRef{Proposition~}{prop:qbgVsDbg}. It is very helpful to compare results about the double Bruhat graph with the much better developed theory of the quantum Bruhat graph.

Under this point of view, one obtains a version of Theorem~\ref{thm:dbgRecursions} for the quantum Bruhat graph. This is a well-known recursive description of weights in the quantum Bruhat graph, cf.\ \cite[Lemma~7.7]{Lenart2015}.
\item
The remainder of this paper will mostly study consequences of recursive relations from Theorem~\ref{thm:dbgRecursions}. By studying the proof of Theorem~\ref{thm:IHProduct} below, one may see that the weight multiset is already uniquely determined by these recursive relations together with a few additional facts to fix a recursive start. This can be seen as an alternative proof that the weight multiset is independent of the chosen reflection order, cf.\ \crossRef{Corollary~}{cor:reflectionOrderInvariance}.
\end{enumerate}
\end{remark}
\section{Iwahori--Hecke algebra}\label{sec:ihalgebra}

Let us briefly motivate the definition of the Iwahori--Hecke algebra associated with an affine Weyl group.

Under suitable assumptions on our group and our fields, the \emph{Hecke algebra} $\mathcal H(G, I)$ is classically defined to be the complex vector space of all compactly supported functions $f: G(F)\rightarrow\mathbb C$ satisfying $f(i_1 g i_2) = f(g)$ for all $g\in G(F), i_1, i_2\in I\cap G(F)$. It becomes an algebra where multiplication is defined via convolution of functions. In this form, it occurs in the classical formulation of the Satake isomorphism \cite{Satake1963}.

It is proved by Iwahori--Matsomoto \cite[Section~3]{Iwahori1965} for split $G$ that $\mathcal H(G, I)$ has a basis given by $\{S_x\mid x\in \widetilde W\}$ over $\mathbb C$ where the multiplication is uniquely determined by the conditions
\begin{align*}
\begin{array}{ll}
S_x S_y = S_{xy},& x,y\in \widetilde W\text{ and } \ell(xy) = \ell(x)+\ell(y),\\
S_{r_a} S_x = qS_{r_a x} + (q-1) S_x,& x\in \widetilde W, a\in \Delta_\af\text{ and }\ell(r_a x) < \ell(x).
\end{array}
\end{align*}
Here, $q := \#\left(\mathcal O_F/\mathfrak m_{\mathcal O_F}\right)$ is the cardinality of the residue field of $F$. The basis element $S_x$ corresponds to the indicator function of the coset $IxI\subseteq G(\breve F)$.

With the convenient change of variables $T_x := q^{-\ell(x)/2} S_x\in \mathcal H(G, I)$, the above relations get the equally popular form
\begin{align*}
\begin{array}{ll}
T_x T_y = T_{xy},& x,y\in \widetilde W\text{ and } \ell(xy) = \ell(x)+\ell(y),\\
T_{r_a} T_x = T_{r_a x} + (q^{1/2}-q^{-1/2}) T_x,& x\in \widetilde W, a\in \Delta_\af\text{ and }\ell(r_a x) < \ell(x).
\end{array}
\end{align*}
Since the number $q$ is independent of the choice of affine root system, we define the \emph{Iwahori--Hecke algebra} of $\widetilde W$ as follows.
\begin{definition}\label{def:IHAlgebra}
The \emph{Iwahori--Hecke algebra} $\mathcal H(\widetilde W)$ of $\widetilde W$ is the algebra over $\mathbb Z[Q]$ defined by the generators
\begin{align*}
T_x,\qquad x\in \widetilde W
\end{align*}
and the relations
\begin{align*}
\begin{array}{ll}
T_x T_y = T_{xy},& x,y\in \widetilde W\text{ and } \ell(xy) = \ell(x)+\ell(y),\\
T_{r_a} T_x = T_{r_a x} + Q T_x,& x\in \widetilde W, a\in \Delta_\af\text{ and }\ell(r_a x) < \ell(x).
\end{array}
\end{align*}
\end{definition}
One easily sees that $\mathcal H(\widetilde W)$ is a free $\mathbb Z[Q]$-module with basis $\{T_x\mid x\in\widetilde W\}$, and that each $T_x$ is invertible, because
\begin{align*}
T_{r_a} (T_{r_a}-Q) = 1,\qquad a\in \Delta_\af.
\end{align*}
All results presented in this article can be immediately generalized to most other conventions for the Iwahori--Hecke algebra, e.g.\ by substituting $Q = q^{1/2}-q^{-1/2}$.
\subsection{Products via the double Bruhat graph}
We are interested in the question how to express arbitrary products of the form $T_x T_y$ with $x, y\in \widetilde W$ in terms of this basis. This is related to understanding the structure of the subset $IxI \cdot IyI\subseteq G(\breve F)$. While it might be too much to ask for a general formula, we can understand these products (and thus the Iwahori--Hecke algebra) better by relating it to the double Bruhat graph.
Our main result of this section is the following:
\begin{theorem}\label{thm:IHProduct}
Let $C_1>0$ be a constant and define $C_2 := (8\#\Phi^++ 4)C_1$.

Let $x=w_x\varepsilon^{\mu_x},z=w_z\varepsilon^{\mu_z}\in \widetilde W$ such that $x$ is $C_2$-regular and $z$ is $2\ell(x)$-regular. Define polynomials $\varphi_{x,z,yz}\in \mathbb Z[Q]$ via
\begin{align*}
T_x T_z = \sum_{y\in \widetilde W} \varphi_{x,z,yz}T_{yz}\in \mathcal H(\widetilde W).
\end{align*}
Pick an element $y=w_y\varepsilon^{\mu_y}\in \widetilde W$ such that $\ell(x)-\ell(y)<C_1$. Let
\begin{align*}
\LP(x) = \{v_x\},\quad \LP(y) = \{v_y\},\quad \LP(z) = \{v_z\}
\end{align*}
and define the multiset
\begin{align*}
M := \left\{\ell_1+\ell_2\mid \begin{array}{l}(\omega_1,\ell_1)\in \wts(v_x\Rightarrow v_y\dashrightarrow w_z v_z),\\
(\omega_2,\ell_2)\in
\wts(w_x v_x w_0\Rightarrow w_y v_y w_0\dashrightarrow w_y w_z v_z)
\\
\text{s.th.\ }v_y^{-1}\mu_y = v_x^{-1}\mu_x-\omega_1+w_0\omega_2\end{array}\right\}_m.
\end{align*}
Here, $w_0\in W$ denotes the longest element. Then
\begin{align*}
\varphi_{x,z,yz} = \sum_{e\in M} Q^e.
\end{align*}
\end{theorem}
\begin{remark}\label{rem:IHProduct}
\begin{enumerate}[(a)]
\item In principle, we have the following recursive relations to calculate $T_x T_z$ as long as all occuring elements are in shrunken Weyl chambers, e.g.\ $2$-regular: Pick a simple affine root $a=(\alpha,k)\in \Delta_\af$. If $xr_a<x$ (i.e.\ $v_x^{-1}\alpha\in \Phi^+$), then
\begin{align*}
T_x T_z = T_{xr_a} T_{r_a} T_z = \begin{cases} T_{xr_a} T_{r_a z},&r_az>z~(\text{i.e.\ }(w_zv_z)^{-1}\alpha\in \Phi^+),\\
T_{xr_a} T_{r_a z} + QT_{x r_a} T_z,&r_az<z~(\text{i.e.\ }(w_zv_z)^{-1}\alpha\in \Phi^-).\end{cases}
\end{align*}
This kind of recursive relation is analogous to the recursive behaviour of the multiset $\wts(v_x\Rightarrow v_y\dashrightarrow w_z v_z)$, cf.\ Theorem \ref{thm:dbgRecursions}.

Similarly, if $r_ax<x$ (i.e.\ $(w_x v_x)^{-1}\alpha\in \Phi^-$), we get
\begin{align*}
T_x T_z &= T_{r_a} T_{r_a x} T_z = \sum_{y\in \widetilde W} \varphi_{r_a x, z, yz} T_{r_a} T_{yz} \\&=  \sum_{y\in \widetilde W} \varphi_{r_a x, z, yz}\cdot\begin{cases} T_{r_a yz},& r_ayz>yz~(\text{i.e.\ }(w_y w_z v_z)^{-1}\alpha\in \Phi^+),\\ T_{r_a yz}+QT_{yz},& r_ayz<yz~(\text{i.e.\ }(w_y w_z v_z)^{-1}\alpha\in \Phi^-).\end{cases}
\end{align*}
This kind of recursive relation is analogous to the recursive behaviour of the multiset $\wts(w_x v_x w_0\Rightarrow w_y v_y w_0\dashrightarrow w_y w_z v_z)$, cf.\  Theorem \ref{thm:dbgRecursions}.

For the proof of Theorem \ref{thm:IHProduct}, we have to apply these recursive relations iteratively while keeping track of the length and regularity conditions to ensure everything happens inside the shrunken Weyl chambers.
\item Let us compare Theorem~\ref{thm:IHProduct} to the quantum Bruhat graph. In view of \crossRef{Proposition }{prop:qbgVsDbg}, it follows that $\varphi_{x,z,yz} = 0$ unless
\begin{align*}
v_y^{-1}\mu_y \leq v_x^{-1}\mu_x - \wt_{\QB(W)}(v_x\Rightarrow v_y) - \wt_{\QB(W)}(w_y v_y\Rightarrow w_x v_x).
\end{align*}
By \cite[Theorem~4.2]{Schremmer2024_bruhat}, this latter inequality is equivalent to the Bruhat order condition $y\leq x$, which is (by definition of the Iwahori Hecke algebra) always a necessary condition for $\varphi_{x,z,yz}$ to be non-zero.
\item If the condition $\ell(x)-\ell(y)<C_1$ gets strengthened to $\ell(x)+\ell(z)-\ell(yz)<C_1$, it follows that the product $yz$ must be length additive, so $v_y = w_z v_z$ \cite[Lemma~2.13]{Schremmer2022_newton}. One of the simple facts on the double Bruhat graph \crossRef{Lemma~}{lem:qbgNonEmptiness} yields
\begin{align*}
\wts(w_x v_x w_0\Rightarrow w_y v_y w_0\dashrightarrow w_y w_z v_z) = \begin{cases}\emptyset,&w_y v_y \neq w_x v_x,\\
\{(0,0)\}_m,&w_y v_y = w_x v_x.
\end{cases}
\end{align*}

So the multiset $M$ as defined in Theorem~\ref{thm:IHProduct} is empty unless $w_y v_y = w_x v_x$, in which case it will be equal to
\begin{align*}
M = \{\ell\mid (\omega,\ell)\in \wts(v_x\Rightarrow v_y)\text{ s.th.\ }v_y^{-1}\mu_y = v_x^{-1}\mu_x-\omega\}_m.
\end{align*}
This recovers Theorem~\ref{thm:introMultiplication}.

The unique smallest element of $\wts(v_x\Rightarrow v_y)$ from \crossRef{Proposition }{prop:qbgVsDbg} corresponds to the uniquely determined largest element in $\widetilde W$ having non-zero coefficient in $T_x T_z$. This element is known as the \emph{Demazure product} of $x$ and $z$ in $\widetilde W$. We recover the formula for the Demazure product of $x$ and $z$ in terms of the quantum Bruhat graph from He--Nie \cite[Proposition~3.3]{He2024_demazure} in the situation of Theorem~\ref{thm:IHProduct}.
\end{enumerate}
\end{remark}
\begin{definition}
\begin{enumerate}[(a)]
\item
For $x\in \widetilde W$ and $w\in W$, we define the multiset $Y(x,w)$ as follows: The underlying set $\abs{Y(x,w)}$ is a subset of $\widetilde W\times\mathbb Z$, and the multiplicity of the pair $(y,e)\in \widetilde W\times \mathbb Z$ in $Y(x,w)$ is defined via the equation
\begin{align*}
T_x T_{w\varepsilon^{2\rho^\vee \ell(x)}} = \sum_{(y,e)\in Y(x,w)} Q^e T_{yw\varepsilon^{2\rho^\vee \ell(x)}}.
\end{align*}
\item We define the usual product group structure on $\widetilde W\times \mathbb Z$, i.e.\
\begin{align*}
(y_1, e_1)\cdot (y_2, e_2) := (y_1 y_2, e_1+e_2)
\end{align*}
for $y_1, y_2\in \widetilde W$ and $e_1, e_2\in \mathbb Z$. If $M$ is a multiset with $\abs M\subseteq \widetilde W\times \mathbb Z$, we write $M\cdot (y,e)$ for the multiset obtained by the right action of $(y,e)\in \widetilde W\times\mathbb Z$.
\end{enumerate}
\end{definition}
\begin{lemma}\label{lem:YsetFacts}
Let $x,z\in \widetilde W$ such that $z$ is $2\ell(x)$-regular.
\begin{enumerate}[(a)]
\item Write $z = w_z\varepsilon^{\mu_z}$ and $\LP(z) = \{v_z\}$. Then
\begin{align*}
T_x T_{z} = \sum_{(y,e)\in Y(x,w_z v_z)} Q^e T_{yz}.
\end{align*}
\item Let $a = (\alpha,k)\in \Delta_\af$ with $xr_a<x$ and $w\in W$. If $w^{-1}\alpha\in \Phi^+$, we have
\begin{align*}
Y(x,w) = Y(x r_a, s_\alpha w)\cdot(r_a,0).
\end{align*}
If $w^{-1}\alpha\in \Phi^-$, we express $Y(x,w)$ as the additive union of multisets
\begin{align*}
Y(x,w) = \Bigr(Y(x r_a, s_\alpha w)\cdot (r_a,0)\Bigl)\cup \Bigl(Y(xr_a, w)\cdot(1,1)\Bigr).
\end{align*}
\item For $y = w_y \varepsilon^{\mu_y}\in \widetilde W$ and $e\in\mathbb Z$, the multiplicity of $(y,e) \in Y(x,w)$ agrees with the multiplicity of $(y^{-1}, e)$ in $Y(x^{-1}, \cl(y)w)$, where $\cl(y)\in W$ is the classical part of $y\in W\ltimes X_\ast(T)_{\Gamma_0}$.
\end{enumerate}
\end{lemma}
\begin{proof}
\begin{enumerate}[(a)]
\item
The regularity condition allows us to write $z$ as the length additive product
\begin{align*}
z = z_1\cdot z_2,\qquad z_1 = w_z v_z\varepsilon^{2\rho^\vee \ell(x)},\qquad z_2 = v_z^{-1}\varepsilon^{\mu_z - v_z 2\rho^\vee \ell(x)}.
\end{align*}
Then we get
\begin{align*}
T_x T_z = T_x T_{z_1} T_{z_2} = \sum_{(y,e)\in Y(x,w_z v_z)} T_{yz_1} T_{z_2}.
\end{align*}
By regularity of $z_1$, it follows that $\LP(yz_1) = \LP(z_1) = \{1\}$ for each $y\leq x$ in the Bruhat order. Thus $T_{yz_1} T_{z_2} = T_{yz_1 z_2} = T_{yz}$ for each $(y,e)\in Y(x,w_z v_z)$.
\item Let $z = w\varepsilon^\mu$ with $\mu$ superregular and dominant, as in (a). Use the fact
\begin{align*}
T_x T_z = T_{xr_a} T_{r_a} T_z
\end{align*}
and evaluate $T_{r_a} T_z$ depending on whether $w^{-1}\alpha$ is positive or negative.
\item We consider the symmetrizing form of $\mathcal H(\widetilde W)$ given by
\begin{align*}
\tau : \mathcal H(\widetilde W)\rightarrow\mathbb Z[Q],\qquad \sum_{x\in\widetilde W} a_x T_x\mapsto a_1.
\end{align*}
One checks that $\tau(T_x T_{x^{-1}}) = 1$ and $\tau(T_x T_y)=0$ for $x,y\in\widetilde W$ with $xy\neq 1$, cf.\ \cite[Section~4.1D]{Bonnafe2017}. It follows from this that $\tau(hh') = \tau(h'h)$ for all $h,h'\in\mathcal H(\widetilde W)$, and that $\tau(T_{x^{-1}} h)$ is the $T_x$-coeffient of $h$ for $x\in\widetilde W$.

Moreover, note that $T_x\mapsto T_{x^{-1}}$ defines an anti-automorphism of the $\mathbb Z[Q]$-algebra $\mathcal H(\widetilde W)$, and that $\tau$ is invariant under this map.

Fix $y\in\widetilde W$ and assume that both $z$ and $yz$ are $2\ell(x)$-regular. We calculate
\begin{align*}
&\sum_{e\in\mathbb Z}\left(\text{multiplicity of }(y,e)\text{ in }Y(x,w_z v_z)\right)Q^e
\\=&\text{(coefficient of $T_{yz}$ in $T_x T_z$)}
\\=&\tau(T_{(yz)^{-1}} T_x T_z)
\\=&\tau(T_{z^{-1}} T_{x^{-1}} T_{yz})
\\=&\text{(coefficient of $T_{z}$ in $T_x^{-1} T_{yz}$)}
\\=&\sum_{e\in\mathbb Z}\left(\text{multiplicity of }(y^{-1},e)\text{ in }Y(x^{-1} ,w_y w_z v_z)\right)Q^e.
\end{align*}
Comparing coefficients of $Q^e$ in $\mathbb Z[Q]$, the claim follows.
\qedhere\end{enumerate}
\end{proof}
\begin{remark}
The connection to our previous article \cite{Schremmer2023_orbits} is given as follows: For $x,z$ as in Lemma~\ref{lem:YsetFacts}, the regularity condition on $z$ basically ensures that $z I z^{-1}$ behaves like $\prescript{w_z v_z}{}{}U(L)$, so we can approximate $IzI$ by the semi-infinite orbit $I z \prescript{v_z}{}{}U(L) = I\prescript{w_z v_z}{}{}U(L)z$. Then $IxI\cdot IzI$ is very close to
\begin{align*}
IxI\cdot \prescript{w_z v_z}{}{}U(L)z = \bigcup_{(y,e)\in Y(x,w_z v_z)} Iy\prescript{w_z v_z}{}{}U(L)z\subseteq G(\breve F).
\end{align*}
Now observe for any $y\in\widetilde W$ that \begin{align*}
IxI\cap I y\prescript{w_z v_z}{}{}U(L)\neq\emptyset\iff y\in IxI\cdot \prescript{w_z v_z}{}{}U(L).\end{align*}
So the multiset $Y(x,w)$ is the representation-theoretic correspondent of the main object of interest in \crossRef{Theorem~}{thm:generalizedMV}.
\end{remark}
\begin{lemma}\label{lem:YsetIdentities}
Let $x = w_x\varepsilon^{\mu_x}\in \widetilde W$ and pick elements $u_1, u_2\in W$ as well as $v_x\in \LP(x)$.
\begin{enumerate}[(a)]
\item The multiset $\wts(v_x\Rightarrow u_1\dashrightarrow u_2)$ is equal to the additive union of multisets
\begin{align*}
\bigcup_{(w_y \varepsilon^{\mu_y},e)\in Y(x,u_2)}\Bigl\{(v_x^{-1}\mu_x - u_1^{-1}\mu_y+\omega,e+\ell)\mid (\omega,\ell)\in \wts(w_x v_x\Rightarrow w_y u_1\dashrightarrow w_y u_2)\Bigr\}_m.
\end{align*}
\item The multiset $\wts(w_x v_x w_0\Rightarrow u_2 w_0 \dashrightarrow u_1)$ is equal to the additive union of multisets
\begin{align*}
\bigcup_{\substack{u_3\in W\\
(w_y \varepsilon^{\mu_y},e)\in Y(x, u_3)\\
\text{s.th.\ }w_yu_3 = u_1}} \Bigl\{&(w_0 u_2^{-1} w_y\mu_y-w_0v_x^{-1}\mu_x + \omega,e+\ell)\\[-1cm]&\quad\mid (\omega,\ell) \in \wts(v_x w_0\Rightarrow w_y^{-1} u_2 w_0\dashrightarrow u_3)\Bigr\}_m.
\end{align*}
\end{enumerate}
\end{lemma}
\begin{proof}
\begin{enumerate}[(a)]
\item Induction on $\ell(x)$. In case $\ell(x)=0$, we get $Y(x,u_2) = \{(x,0)\}_m$. From \crossRef[ (c)]{Lemma }{lem:pathSymmetries}, we indeed get that $\wts(v_x\Rightarrow u_1\dashrightarrow u_2)$ is equal to
\begin{align*}
\{(v_x^{-1}\mu_x - u_1^{-1}\mu_x + \omega,\ell)\mid (\omega,\ell)\in \wts(w_x v_x\Rightarrow w_x u_1\dashrightarrow w_x u_2)\}_m.
\end{align*}
Now in the inductive step, pick a simple affine root $a = (\alpha,k)$ with $xr_a<x$.
This means $v_x^{-1}\alpha\in \Phi^+$ and $v_{x'} := s_\alpha v_x\in \LP(x')$, where 
\begin{align*}
x' := w_{x'}\varepsilon^{\mu_{x'}} := xr_a = w_xs_\alpha\varepsilon^{s_\alpha(\mu_x) + k\alpha^\vee}.
\end{align*}
Let us first consider the case $u_2^{-1}\alpha \in \Phi^+$. Then $Y(x,u_2) = Y(x', s_\alpha u_2)\cdot (r_a,0)$ by Lemma \ref{lem:YsetFacts} (b). We get
\begin{align*}
&\bigcup_{(w_y \varepsilon^{\mu_y},e)\in Y(x,u_2)}\Bigl\{(v_x^{-1}\mu_x - u_1^{-1}\mu_y+\omega,e+\ell)\mid (\omega,\ell)\in \wts(w_x v_x\Rightarrow w_y u_1\dashrightarrow w_y u_2)\Bigr\}_m
\\=&\bigcup_{(w_{y'} \varepsilon^{\mu_{y'}},e)\in Y(x',s_\alpha u_2)}\Bigl\{(v_{x'}^{-1}\mu_x' + kv_x^{-1} \alpha^\vee - (s_\alpha u_1)^{-1}\mu_{y'} - k u_1^{-1}\alpha^\vee+\omega,e+\ell)\mid \\&\hspace{4cm}(\omega,\ell)\in \wts(w_{x'} v_{x'}\Rightarrow w_{y'} (s_\alpha u_1)\dashrightarrow w_{y'} (s_\alpha u_2))\Bigr\}_m.
\end{align*}
By the inductive assumption, this is equal to
\begin{align*}
\{(\omega + k(v_x^{-1}\alpha^\vee - u_1^{-1}\alpha^\vee),\ell)\mid (\omega,\ell)\in \wts(s_\alpha v_x\Rightarrow s_\alpha u_1\dashrightarrow s_\alpha u_2)\}_m.
\end{align*}
By Theorem \ref{thm:dbgRecursions} (a), this is equal to $\wts(v_x\Rightarrow u_1\dashrightarrow u_2)$, using the assumption $u_2^{-1}\alpha\in \Phi^+$ again.

In the converse case where $u_2^{-1}\alpha\in \Phi^-$, we argue entirely similarly. Use Lemma~\ref{lem:YsetFacts} to write\begin{align*}
Y(x,u_2) = \Bigl(Y(x', s_\alpha u_2)\cdot (r_a,0)\Bigr) \cup \Bigl(Y(x', u_2)\cdot(1,1)\Bigr)
\end{align*}
Considering Theorem \ref{thm:dbgRecursions} (b), the inductive claim follows.
\item One may argue similarly to (a), tracing through somewhat more complicated expressions to reduce to Theorem~\ref{thm:dbgRecursions} again. Instead, we show that (a) and (b) are equivalent. Recall that $w_x v_x w_0\in \LP(x^{-1})$ \cite[Lemma~2.12]{Schremmer2022_newton}. By (a), we see that $\wts(w_x v_x w_0\Rightarrow u_2 w_0\dashrightarrow u_1)$ is equal to
\begin{align*}
\bigcup_{(w_y \varepsilon^{\mu_y},e)\in Y(x^{-1}, u_1)}\Bigl\{&\quad((w_x v_x w_0)^{-1}(-w_x\mu_x) - (u_2 w_0)^{-1} \mu_y + \omega,e+\ell),
\\[-0.5cm]&\mid (\omega,\ell)\in \wts(v_x w_0\Rightarrow w_y u_2 w_0\dashrightarrow w_y u_1)\Bigr\}.
\end{align*}
In view of Lemma \ref{lem:YsetFacts} (c), we recover the claim in (b).
\qedhere\end{enumerate}
\end{proof}
\begin{lemma}\label{lem:IHProduct}
Let $C_1,e\geq 0$ be two non-negative integers. Define $C_2 := (8e+4) C_1$.

Let $x,y\in \widetilde W$ such that $x$ is $C_2$-regular and $\ell(x)-\ell(y)< C_1$. Let $u\in W$. Write
\begin{align*}
\begin{array}{lll}
x = w_x\varepsilon^{\mu_x},& y = w_y \varepsilon^{\mu_y},\\
\LP(x) = \{v_x\},&\LP(y) = \{v_y\}.
\end{array}
\end{align*}

Define the multiset
\begin{align*}
M := \left\{\ell_1+\ell_2\mid \begin{array}{l}(\omega_1, \ell_1) \in \wts(v_x\Rightarrow v_y \dashrightarrow u),\\(\omega_2, \ell_2)\in \wts(w_x v_xw_0\Rightarrow w_y v_yw_0\dashrightarrow w_y u)\\\text{ s.th.\ }v_y^{-1}\mu_y = v_x^{-1}\mu_x - \omega_1+w_0\omega_2\end{array}\right\}_m.
\end{align*}
Then the multiplicity of $(y,e)$ in $Y(x,u)$ agrees with the multiplicity of $e$ in $M$.
\end{lemma}
\begin{proof}
Induction on $e$. Consider the inductive start $e=0$. If $0 \in M$, then $\ell_1=\ell_2=0$ and $v_x=v_y$ by definition of $M$. Hence $x=y$, and indeed $0\in M$ has multiplicity $1$. Similarly, $(y,0)$ also has multiplicity $1$ in $Y(x,u)$.

If $0\notin M$, we see $x\neq y$ and indeed $(y,0)\notin Y(x,u)$ for $x\neq y$. This settles the inductive start.

In the inductive step, let us write $x$ as length additive product $x = x_1 x_2 x_3$ where
\begin{align*}
x_1 = \varepsilon^{4C_1 w_x v_x \rho^\vee},\qquad x_2 = w_x\varepsilon^{\mu_x - 8C_1 v_x \rho^\vee},\qquad x_3 = \varepsilon^{4C_1v_x \rho^\vee}.
\end{align*}
Note that the inductive assumptions are satisfied for $C_1, e-1,x_2$ and any element $y'\in \widetilde W$ such that $\ell(x_2)-\ell(y')<C_1$.

The length additivity of $x = x_1 x_2 x_3$ implies
\begin{align*}
Y(x,u) = \Bigl\{(y_1 y_2 y_3, e_1+e_2+e_3)\mid \begin{array}{l}(y_3, e_3)\in Y(x_3, u),\\(y_2, e_2)\in Y(x_2, \cl(y_3)u),\\(y_1,e_1)\in Y(x_1, \cl(y_2)\cl(y_3)u)\end{array}\Bigr\}_m.
\end{align*}
Pick elements
\begin{align*}
(y_3, e_3)\in Y(x_3, u),\quad (y_2, e_2)\in Y(x_2, \cl(y_3)u),\quad (y_1, e_1)\in Y(x_1, \cl(y_2)\cl(y_3)u)
\end{align*}
such that $\ell(y_1 y_2 y_3)>\ell(x)-C_1$ and $e_1+e_2+e_3=e$.

In this case, we certainly get $\ell(y_i)>\ell(x_i)-C_1$ for $i=1,2,3$. Since $x_1, x_2, x_3$ are $4C_1$-regular by construction, it follows that each $y_i$ is $2C_1$-regular by $y_i\leq x_i$ and $\ell(y_i)>\ell(x_i)-C_1$ (studying how regularity behaves in a sequence of Bruhat covers from $y_i$ to $x_i$). We claim that
\begin{align*}
\ell(y_1y_2y_3)=\ell(y_1)+\ell(y_2)+\ell(y_3).
\end{align*}
We can study the question of length additivity of such products using \cite[Lemma~2.13]{Schremmer2022_newton}. This lemma expresses the condition $\ell(xy) = \ell(x)+\ell(y)$ in terms of the \emph{length functionals} $\ell(x,\cdot)$ and $\ell(y,\cdot)$ as defined in \cite[Definition~2.5]{Schremmer2022_newton}. Using the aforementioned lemma, it suffices to see that $\ell(y_1 y_2) = \ell(y_1)+\ell(y_2)$ and $\ell(y_2 y_3)= \ell(y_2)+\ell(y_3)$ (using regularity). If $y_1 y_2$ is not a length additive product, we use the \cite[Lemma~2.13]{Schremmer2022_newton} to find a root $\alpha\in \Phi$ with $\ell(y_1,\cl(y_2)\alpha)>0$ and $\ell(y_2,\alpha)<0$. By regularity, this means $\ell(y_1,\cl(y_2)\alpha)>C_1$ and $\ell(y_2,\alpha)<-C_1$. Using \cite[Corollary~2.10 and Lemma~2.12]{Schremmer2022_newton}, we get \begin{align*}
\ell(y_1 y_2) &= \sum_{\beta\in \Phi} \frac 12\abs{\ell(y_1,\cl(y_2)\beta)+\ell(y_2,\beta)}\\&\leq - C_1 + \sum_{\beta\in \Phi} \frac 12(\abs{\ell(y_1,\cl(y_2)\beta)} + \abs{\ell(y_2,\beta)}) =\ell(y_1)+\ell(y_2)-C_1.
\end{align*}
This contradicts the above assumption $\ell(y_1 y_2 y_3)>\ell(x)-C_1\geq \ell(y_1)+\ell(y_2)+\ell(y_3)-C_1$. The proof that $y_2 y_3$ is length additive is completely analogous.

Let us consider the special case $e_1=e_3=0$ separately. Then $y_1=x_1$ and $y_3=x_3$. The length additivity of the product $x_1 y_2 x_3$ implies that $\LP(y_2) = \{v_x\}$ and $\cl(y_2) = w_x$. Using Lemma \ref{lem:YsetIdentities} (a), we can express $\{(0,0)\}_m = \wts(v_x\Rightarrow v_x\dashrightarrow u)$ in the form
\begin{align*}
\bigcup_{(w_y\varepsilon^{\mu_y},e')\in Y(x_2, u)}\{(\cdots,e'+\ell)\mid (\omega,\ell)\in \wts(w_x v_x\Rightarrow w_y v_x\dashrightarrow w_y u)\}_m.
\end{align*}
From this and \crossRef{Lemma~}{lem:qbgNonEmptiness}, it follows that $Y(x_2, u)$ contains only one element $(y',e')$ with $\cl(y') = w_x$, and that this element must be equal to $(x_2,0)$.

We see that, if $e_1=e_3=0$, we must also have $e_2=0$. This case has been settled before.

We hence assume that $e_1+e_3>0$. In particular, we may apply the inductive assumption to $x_2, y_2, e_2$. Recall that the multiplicity of $(y,e)$ in $Y(x,u)$ is equal to the number of tuples (with multiplicity)
\begin{align*}
(y_3, e_3)\in Y(x_3, u),\quad (y_2,e_2)\in Y(x_2, \cl(y_3)u),\quad (y_1, e_1)\in Y(x_1, \cl(y_2)\cl(y_3)u)
\end{align*}
such that $e_1+e_2+e_3=e$ and $y = y_1y_2 y_3$ (necessarily length additive). Hence $\LP(y_2) = \{\cl(y_3) v_y\}$ and $w_y = \cl(y_1)\cl(y_2)\cl(y_3)$. By induction, the multiplicity of $(y,e)$ in $Y(x,u)$ is also equal to the number of tuples (with multiplicity)
\begin{align*}
(y_3, e_3)\in& Y(x_3, u),\\(\omega_1, \ell_1)\in& \wts(v_x\Rightarrow \cl(y_3)v_{y}\dashrightarrow \cl(y_3) u),\\
(\omega_2, \ell_2)\in &\wts(w_x v_x w_0\Rightarrow \cl(y_1)^{-1}w_yv_{y}w_0\dashrightarrow \cl(y_1)^{-1}w_yu),
\\(y_1, e_1)\in& Y(x_1, \cl(y_1)^{-1}w_yu),
\end{align*}
satisfying $e = e_1+\ell_1+\ell_2+e_3$ and
\begin{align*}
y_1^{-1} y y_3^{-1} = \cl(y_1)^{-1} w_y \cl(y_3)^{-1}\varepsilon^{ (\cl(y_3) v_y)(v_x^{-1}\mu_{x_2} -\omega_1+w_0 \omega_2)}.
\end{align*}
The latter identity can be rewritten, if we write $y_3 = w_3\varepsilon^{\mu_3}$ and $y_1 = w_1\varepsilon^{\mu_1}$, as
\begin{align*}
v_y^{-1}\mu_y = v_x^{-1}\mu_{x_2} - \omega_1+w_0\omega_2 + v_y^{-1}\mu_3 + (w_yv_y)^{-1}w_1\mu_1.
\end{align*}
We see that we may study the contributions of $(y_3, e_3, \omega_1,\ell_1)$ and $(y_1, e_1, \omega_2, \ell_2)$ separately.

We may combine the above data for $(y_3, e_3, \omega_1,\ell_1)$, noticing that we are only interested in the multiset
\begin{align*}
\Bigl\{(-v_y^{-1}\mu_3+\omega_1 + v_x^{-1}\mu_{x_3}, e_3+\ell_1)\mid \begin{array}{c}(w_3\varepsilon^{\mu_3},e_3)\in Y(x_3,u),\\
(\omega_1,\ell_1)\in \wts(v_x\Rightarrow w_3 v_y\dashrightarrow w_3u)\end{array}\Bigr\}_m.
\end{align*}
By Lemma \ref{lem:YsetIdentities} (a), the above multiset agrees with $\wts(v_x\Rightarrow v_y\dashrightarrow u)$.

Similarly, we may combine the data for $(y_1, e_1, \omega_2, \ell_2)$, noticing that we are only interested in the multiset
\begin{align*}
\Bigl\{&(w_0(w_y v_y)^{-1}w_1\mu_1+\omega_2 - w_0 (w_xv_x)^{-1}\mu_{x_1}, e_1+\ell_2)\mid\\&\quad \begin{array}{c}u'\in W\\(w_1\varepsilon^{\mu_1},e_1)\in Y(x_1,u')\text{ s.th.\ }w_1 u' = w_yu\\
(\omega_2,\ell_2)\in \wts(w_xv_xw_0 \Rightarrow w_1^{-1}w_y v_yw_0 \dashrightarrow w_1^{-1} w_yu)\end{array}\Bigr\}_m.
\end{align*}
By Lemma \ref{lem:YsetIdentities} (b), the above multiset agrees with $\wts(w_x v_x w_0\Rightarrow w_y v_y w_0\dashrightarrow w_y u)$.

We summarize that the multiplicity of $(y,e)$ in $Y(x,u)$, i.e.\ the number of tuples $(y_3,e_3, \omega_1,\ell_1,\omega_2,\ell_2,y_1,e_1)$ with multiplicity as above, is equal to the number of tuples
\begin{align*}
(\lambda_1, f_1)\in&\wts(v_x\Rightarrow v_y\dashrightarrow u)
\\(\lambda_2,f_2)\in&\wts(w_x v_x w_0\Rightarrow w_y v_y w_0\dashrightarrow w_y u)
\end{align*}
satisfying $e = f_1 + f_2$ and
\begin{align*}
v_y^{-1}\mu_y = v_x^{-1} \mu_{x_2} - \lambda_1 + v_x^{-1}\mu_{x_3} + w_0\lambda_2 + (w_x v_x)^{-1}\mu_{x_1}.
\end{align*}
Up to evaluating the product $x=x_1 x_2 x_3 \in W\ltimes X_\ast(T)_{\Gamma_0}$, this finishes the induction and the proof.
\end{proof}
\begin{corollary}
\label{cor:productUpperBound}
Let $x = w_x\varepsilon^{\mu_x}, z = w_z\varepsilon^{\mu_z}\in \widetilde W$. Write
\begin{align*}
T_{x} T_{z} = \sum_{y\in \widetilde W}\sum_{e\geq 0} n_{y,e} Q^e T_{yz},\qquad n_{y,e}\in\mathbb Z_{\geq 0}.
\end{align*}
Pick elements $v_x\in \LP(x_x),v_z\in \LP(x_z),e\in\mathbb Z_{\geq 0}$ and $y = w_y \varepsilon^{\mu_y} \in \widetilde W$. Then $n_{y,e}$ is at most equal to the multiplicity of the element
\begin{align*}
\left(v_x^{-1}\left(\mu_x - w_x^{-1} w_y\mu_y\right), e\right)
\end{align*}
in the multiset
\begin{align*}
\wts(v_x\Rightarrow w_y^{-1} w_x v_x\dashrightarrow w_z v_z).
\end{align*}
\end{corollary}
\begin{proof}
Let us write $\mathcal H(\widetilde W)_{\geq 0}$ for the subset of those elements of $\mathcal H(\widetilde W)$ which are non-negative linear combinations of elements of the form $Q^e T_x$ for $e\in \mathbb Z_{\geq 0}$ and $x\in \widetilde W$.
For dominant coweights $\lambda_1, \lambda_2\in X_\ast(T)_{\Gamma_0}$, we obtain
\begin{align*}
&T_{\varepsilon^{w_x v_x \lambda_1} x} T_{z\varepsilon^{v_z \lambda_2}} = T_{\varepsilon^{w_x v_x \lambda_1}} T_x T_z T_{\varepsilon^{v_z\lambda_2}}
\\=&\sum_{y\in \widetilde W}\sum_{e\geq 0} n_{y,e} Q^e T_{\varepsilon^{w_x v_x \lambda_1}}T_{yz}T_{\varepsilon^{v_z\lambda_2}}
\\\in&\sum_{y\in \widetilde W}\sum_{e\geq 0} n_{y,e} Q^e T_{\varepsilon^{w_x v_x \lambda_1}yz\varepsilon^{v_z\lambda_2}}+\mathcal H(\widetilde W)_{\geq 0}.
\end{align*}
So the quantity $n_{y,e}$ can only increase if we replace $(x,y,z)$ by $(\varepsilon^{w_x v_x\lambda_1} x,\varepsilon^{w_x v_x \lambda_1} y,z\varepsilon^{w_z v_z\lambda_2})$. Choosing our dominant coweights $\lambda_1, \lambda_2$ appropriately regular, the claim follows from Lemma~\ref{lem:IHProduct}.
\end{proof}
\begin{proof}[Proof of Theorem~\ref{thm:IHProduct}]
In view of Corollary~\ref{cor:productUpperBound} and the definition of paths in the double Bruhat graph, it follows easily that for all $x,y,z\in \widetilde W$, the degree of the polynomial $\varphi_{x,y,z}$ in $\mathbb Z[Q]$ is bounded from above by $\#\Phi^+$ (re-proving this well-known fact). Thus the theorem follows by assuming $e\leq \#\Phi^+$ in Lemma~\ref{lem:IHProduct} (noticing that also the multiset $M$ cannot contain elements $>\#\Phi^+$ using the definition of paths in the double Bruhat graph).
\end{proof}
\subsection{Class polynomial}
Choose for each $\sigma$-conjugacy class $\mathcal O\subseteq \widetilde W$ a minimal length element $x_{\mathcal O}\in \mathcal O$. Then the class polynomials associated with each $x\in\widetilde W$ are the uniquely determined polynomials $f_{x,\mathcal O}\in \mathbb Z[Q]$ satisfying
\begin{align*}
T_x \equiv \sum_{\mathcal O} f_{x,\mathcal O} T_{x_{\mathcal O}}\pmod{[\mathcal H,\mathcal H]_\sigma},
\end{align*}
where $[\mathcal H,\mathcal H]_\sigma$ is the $\mathbb Z[Q]$-submodule of $\mathcal H$ generated by the elements of the form
\begin{align*}
[h,h']_\sigma = hh' - h'\sigma(h)\in\mathcal H.
\end{align*}
These polynomials $f_{x,\mathcal O}\in\mathbb Z[Q]$ are independent of the choice of minimal length representatives $x_{\mathcal O}\in\mathcal O$, and there is an explicit algorithm to compute them, cf.\ \cite{He2014b}. Using this algorithm, one easily sees the following boundedness property: Whenever $\ell(x)<\ell(x_{\mathcal O})$, we must have $f_{x,\mathcal O}=0$. The main result of this section is the following.

\begin{theorem}\label{thm:classPolynomialViaDbg}
Let $B>0$ be any real number. There exists an explicitly described constant $B'>0$, depending only on $B$ and the root system $\Phi$, such that the following holds true:

Let $x = w\varepsilon^\mu\in \widetilde W$ be $B'$-regular and write $\LP(x) = \{v\}$. For each $\sigma$-conjugacy class $\mathcal O\subseteq\widetilde W$ with $\langle v^{-1}\mu-\nu(\mathcal O),2\rho\rangle\leq B$ and $\kappa(\mathcal O) = \kappa(x)$, we have
\begin{align*}
f_{x,\mathcal O} = \sum_{\substack{(\omega,e)\in \wts(v\Rightarrow\sigma(wv))\text{ s.t.}\\ \nu(\mathcal O) = \avg_\sigma(v^{-1}\mu-\omega)}} Q^e\in\mathbb Z[Q].
\end{align*}
\end{theorem}
\begin{remark}\label{rem:classPolynomials}
\begin{enumerate}[(a)]
\item Our proof reduces Theorem~\ref{thm:classPolynomialViaDbg} to Theorem~\ref{thm:IHProduct}. This yields a short and instructive proof, but results in a very large value of $B'$. One may alternatively compare the aforementioned algorithm of He--Nie directly with Theorem~\ref{thm:dbgRecursions} to obtain a significantly smaller value of $B'$.
\item Explicit formulas for the full class polynomials, rather than just degree and sometimes leading coefficients, have been very rare in the past. One exception of this are the elements with finite Coxeter part as studied in \cite{He2022}. In the setting of Theorem~\ref{thm:classPolynomialViaDbg}, this means that $v^{-1}\sigma(wv)\in W$ has a reduced expression in $W$ where every occurring simple reflection lies in a different $\sigma$-orbit in $S$. Then the class polynomial from \cite[Theorem~7.1]{He2022} is, translating to our notation as above, given by $Q^{\ell(v^{-1}\sigma(wv))}$.

Write $v^{-1}\sigma(wv) = s_{\alpha_1}\cdots s_{\alpha_n}$ for such a reduced expression as above, and choose a reflection order $\prec$ with $\alpha_1\prec\cdots\prec\alpha_n$. Then one sees that there is only one unlabelled $\prec$-increasing path from $v$ to $\sigma(wv)$ in the double Bruhat graph, given by
\begin{align*}
v\rightarrow v s_{\alpha_1}\rightarrow\cdots\rightarrow vs_{\alpha_1}\cdots s_{\alpha_n} = \sigma(wv).
\end{align*}
This path has length $n$. Since the simple coroots $\alpha_1,\dotsc,\alpha_n$ lie in pairwise distinct $\sigma$-orbits, it follows for any coroot $\omega \in\mathbb Z\Phi^\vee$ that there is at most one choice of integers $m_1,\dotsc,m_n\in\mathbb Z$ with
\begin{align*}
m_1\alpha_1^\vee+\cdots+m_n\alpha_n^\vee\equiv \omega \in X_\ast(T)_{\Gamma}.
\end{align*}
With a bit of bookkeeping, one may explicitly describe $\wts(v\Rightarrow\sigma(wv))$ as a multiset of pairs $(\omega,n)$, each with multiplicity one, for exactly those coweights $\omega$ which are non-negative linear combinations of the simple coroots $\alpha_1^\vee,\dotsc,\alpha_n^\vee$. This easy double Bruhat theoretic calculation recovers \cite[Theorem~7.1]{He2022} in the setting of Theorem~\ref{thm:classPolynomialViaDbg}.
\item Let $J\subseteq \Delta$ be the support of $v^{-1}\sigma(wv)$ in $W$. Let $v^J\in W^J$ be the unique minimal length element in $vJ$. Write $v = v^J v_1$ and $\sigma(wv) = v^J v_2$ so that $v_1, v_2\in W_J$. Choosing a suitable reflection order, we get a one to one correspondence between paths in the double Bruhat graph of $W$ from $v$ to $\sigma(wv)$ and paths in the double Bruhat graph of $W_J$ from $v_1$ to $v_2$. The resulting statement on class polynomials recovers \cite[Theorem~C]{He2015b} in the setting of Theorem~\ref{thm:classPolynomialViaDbg}.
\end{enumerate}
\end{remark}
\begin{proof}[Proof of Theorem~\ref{thm:classPolynomialViaDbg}]
Define $C_1 := B+1$, and let $C_2>0$ be as in Theorem~\ref{thm:IHProduct}.

By choosing $B'$ appropriately, we may assume that we can write $x$ as a length-additive product
\begin{align*}
x = x_1 x_2,\qquad x_1 = wv\varepsilon^{\mu_1},\qquad x_2 = v^{-1}\varepsilon^{\mu_2}
\end{align*}
such that $x_1$ is $2\ell(x_2)$-regular and $x_2$ is $C_2$-regular. Observe that $\LP(x_2) = \{v\}$ and $\LP(x_1) = \{1\}$.
Then
\begin{align*}
T_x = T_{x_1} T_{x_2}\equiv T_{x_2}\sigma(T_{x_1})\pmod{[\mathcal H, \mathcal H]_\sigma}.
\end{align*}
Write $\mathcal H_{\leq \ell(x)-B-1}$ for the $\mathbb Z[Q]$-submodule of $\mathcal H$ generated by all elements $T_z$ satisfying $\ell(z)<\ell(x)-B$.

Using Theorem~\ref{thm:introMultiplication}, we may write
\begin{align*}
T_{x_2}T_{\sigma(x_1)} \equiv \sum_{(\omega,e)\in\wts(v\Rightarrow\sigma(wv))} Q^e T_{\varepsilon^{v^{-1}\mu-\omega}}\pmod{\mathcal H_{\leq \ell(x)-B-1}}.
\end{align*}
So if $\mathcal O$ satisfies $\langle v^{-1}\mu-\nu(\mathcal O),2\rho\rangle\leq B$, we see that
\begin{align*}
f_{x,\mathcal O} = \sum_{(\omega,e)\in\wts(v\Rightarrow\sigma(wv))} Q^e f_{\varepsilon^{v^{-1}\mu-\omega},\mathcal O}.
\end{align*}
Here, we used the above observation that $f_{y,\mathcal O}=0$ if $\ell(y)<\langle \nu(\mathcal O),2\rho\rangle$. By regularity of $v^{-1}\mu$ with respect to $\omega$, we see that $v^{-1}\mu-\omega$ is always dominant and $1$-regular in the above sum. Hence
\begin{align*}
f_{\varepsilon^{v^{-1}\mu-\omega},\mathcal O} = \begin{cases}1,&\text{if }\nu(\mathcal O) = \avg_\sigma(v^{-1}\mu-\omega),\\
0,&\text{otherwise}.\end{cases}
\end{align*}
The claim follows.
\end{proof}

\section{Affine Deligne--Lusztig varieties}\label{sec:ADLV}
One crucial feature of the class polynomials $f_{x,\mathcal O}$ is that they encode important information on the geometry of affine Deligne--Lusztig varieties.
\begin{theorem}[{\cite[Theorem~2.19]{He2015}}]\label{thm:adlvViaClassPolynomials}
Let $x\in\widetilde W$ and $[b]\in B(G)$. Denote
\begin{align*}
f_{x,[b]} := \sum_{\mathcal O} Q^{\ell(\mathcal O)} f_{x,\mathcal O}\in\mathbb Z[Q],
\end{align*}
where the sum is taken over all $\sigma$-conjugacy classes $\mathcal O\subset \widetilde W$ whose image in $B(G)$ is $[b]$. For each such $\sigma$-conjugacy class $\mathcal O$, we write
\begin{align*}
\ell(\mathcal O) = \min\{\ell(y)\mid y\in\mathcal O\}.
\end{align*}
Then $X_x(b)\neq\emptyset$ if and only if $f_{x,[b]}\neq 0$. In this case,
\begin{align*}
\dim X_x(b) = \frac 12\left(\ell(x) + \deg(f_{x,[b]})\right)-\langle \nu(b),2\rho\rangle
\end{align*}
and the number of $J_b(F)$-orbits of top dimensional irreducible components in $X_x(b)$ is equal to the leading coefficient of $f_{x,[b]}$.\rightqed
\end{theorem}
Combining with the explicit description of class polynomials from Theorem~\ref{thm:classPolynomialViaDbg}, we conclude the following.
\begin{proposition}\label{prop:ADLVSuperregular}
Let $B>0$ be any real number. There exists an explicitly described constant $B'>0$, depending only on $B$ and the root system $\Phi$, such that the following holds true:

Let $x = w\varepsilon^\mu\in \widetilde W$ be $B'$-regular and write $\LP(x) = \{v\}$. Let $[b]\in B(G)$ such that $\langle v^{-1}\mu-\nu(b),2\rho\rangle<B$ and $\kappa(b) = \kappa(x)$. Let $E$ denote the multiset
\begin{align*}
E = \{e\mid (\omega,e)\in \wts(v\Rightarrow\sigma(wv))\text{ s.th.\ }\nu(b) = \avg_\sigma(v^{-1}\mu-\omega)\}_m.
\end{align*}
Then $X_x(b)\neq\emptyset$ if and only if $E\neq\emptyset$. In this case, set $e := \max(E)$. Then
\begin{align*}
\dim X_x(b) = \frac 12\left(\ell(x)+e-\langle \nu(b),2\rho\rangle\right),
\end{align*}
and the number of $J_b(F)$-orbits of top dimensional irreducible components of $X_x(b)$ is equal to the multiplicity of $e$ in $E$.
\end{proposition}
\begin{proof}
Let $\mathcal O\subseteq \widetilde W$ be the unique $\sigma$-conjugacy class whose image in $B(G)$ is $[b]$ (unique by regularity). Then $\ell(\mathcal O) = \langle \nu(b),2\rho\rangle$ and $f_{x,[b]} = Q^{\ell(\mathcal O)} f_{x,\mathcal O}$. Expressing
\begin{align*}
f_{x,\mathcal O} = \sum_{e\in E} Q^e
\end{align*}
using Theorem~\ref{thm:classPolynomialViaDbg}, the statements follow immediately using Theorem~\ref{thm:adlvViaClassPolynomials}.
\end{proof}
For split groups, this recovers \crossRef{Corollary~}{cor:superregularADLV} up to possibly different regularity constraints. In practice, one may use Proposition~\ref{prop:ADLVSuperregular} to deduce statements on the double Bruhat graph from the well-studied theory of affine Deligne--Lusztig varieties.
\begin{corollary}\label{cor:wtsSpecialEstimates}
Let $u,v\in W$ and let $J=\supp(u^{-1} v)\subseteq \Delta$ be the support of $u^{-1} v$ in $W$, and $\omega\in\mathbb Z\Phi^\vee$.
\begin{enumerate}[(a)]
\item Suppose that $\ell(u^{-1} v)$ is equal to $d_{\QB(W)}(u\Rightarrow v)$, the length of a shortest path from $u$ to $v$ in the quantum Bruhat graph. Then $(\omega,\ell(u^{-1} v))\in \wts(u\Rightarrow v)$ whenever $\omega\geq \wt_{\QB(W)}(u\Rightarrow v)$ and $\omega\in \mathbb Z\Phi_J^\vee$.
\item If $\omega\in \mathbb Z\Phi_J^\vee$ with $\omega \geq 2\rho_J^\vee$, which denotes the sum of positive coroots in $\Phi_J^\vee$, we have
\begin{align*}
(\omega,\ell(u^{-1} v))\in \wts(u\Rightarrow v).
\end{align*}
\end{enumerate}
\end{corollary}
\begin{proof}
Assume without loss of generality that the group $G$ is split. Reducing to the double Bruhat graph of $W_J$ as in Remark \ref{rem:classPolynomials} (d), we may and do assume that $J=\Delta$.

Let $B = \langle \omega,2\rho\rangle+1$ and $B'>0$ as in Proposition~\ref{prop:ADLVSuperregular}. Choose $x = w\varepsilon^\mu\in\widetilde W$ to be $B'$-superregular such that $\LP(x) = \{u\}$ and $v = wu$. Let $[b]\in B(G)$ be the $\sigma$-conjugacy class containing $\varepsilon^{u^{-1}\mu-\omega}$, so that $\nu(b) = u^{-1}\mu-\omega$.
\begin{enumerate}[(a)]
\item By \cite[Proposition~4.2]{Milicevic2020}, the element $x$ is cordial. By \cite[Theorem~1.1]{Milicevic2020} and \cite[Theorem~B]{Goertz2015}, we get $X_x(b)\neq\emptyset$ and
\begin{align*}
\dim X_x(b) = \frac 12\left(\ell(x)+\ell(u^{-1} v) - \langle \nu(b),2\rho\rangle\right).
\end{align*}
The claim follows.
\item Similar to (a), using \cite[Theorem~1.1]{He2021a}. This celebrated result of He shows that if $\omega\geq 2\rho^\vee$ and $\supp(u^{-1} v) = \Delta$, then $X_x(b)\neq\emptyset$ and
\begin{align*}
\dim X_x(b) = \frac 12\left(\ell(x)+\ell(u^{-1} v) - \langle \nu(b),2\rho\rangle\right).
\end{align*}
The claim follows again.
\qedhere\end{enumerate}
\end{proof}
The reader who wishes to familiarize themselves more with the combinatorics of double Bruhat graphs may take the challenge and prove the above corollary directly.

We now want to state the main result of this section, describing the nonemptiness pattern and dimensions of affine Deligne--Lusztig varieties associated with sufficiently regular elements $x\in\widetilde W$ and arbitrary $[b]\in B(G)$. We let $\lambda(b)\in X_\ast(T)_{\Gamma}$ be the $\lambda$-invariant as introduced by Hamacher--Viehmann \cite[Section~2]{Hamacher2018}. By $\conv : X_\ast(T)_{\Gamma}\rightarrow X_\ast(T)_{\Gamma_0}\otimes\mathbb Q$, we denote the convex hull map from \cite[Section~3.1]{Schremmer2022_newton}, so that $\nu(b) = \conv(\lambda(b))$.

Our regularity condition is given as follows: Decompose the (finite) Dynkin diagram of $\Phi$ into its connected components, so we have $\Phi = \Phi_1\sqcup\cdots\sqcup \Phi_c$. Denote $\theta_i\in \Phi_i^+$ the uniquely determined highest root, and write it as linear combination of simple roots
\begin{align*}
\theta_i = \sum_{\alpha\in \Delta} c_{i,\alpha} \alpha.
\end{align*}
Define the regularity constant $C$ to be
\begin{align*}
C = 1+\max_{i=1,\dotsc,c} \sum_{\alpha\in \Delta} c_{i,\alpha}\in\mathbb Z.
\end{align*}
With that, we can state our main result as follows.
\begin{theorem}\label{thm:adlvDimension}
Let $x=w\varepsilon^{\mu}\in\widetilde W$ be $C$-regular and $[b]\in B(G)$ such that $\kappa(b) = \kappa(x)$. Write $\LP(x) = \{v\}$ and denote $E$ to be either of the following two sets $E_1$ or $E_2$:
\begin{align*}
E_1 := &\{e\mid (\omega,e)\in\wts(v\Rightarrow\sigma(wv))\text{ s.th.\ }\lambda(b) \equiv v^{-1}\mu-\omega\in X_\ast(T)_{\Gamma}\},\\
E_2 := &\{e\mid (\omega,e)\in\wts(v\Rightarrow\sigma(wv))\text{ s.th.\ }\nu(b) =\conv(v^{-1}\mu-\omega)\}.
\end{align*}
Then $X_x(b)\neq\emptyset$ if and only if $E\neq\emptyset$. In this case,
\begin{align*}
\dim X_x(b) = \frac 12\left(\ell(x)+\max(E)-\langle \nu(b),2\rho\rangle - \defect(b)\right).
\end{align*}
\end{theorem}
\begin{remark}
\begin{enumerate}[(a)]
\item Since $\conv(\lambda(b))=\nu(b)$, we have $E_1\subseteq E_2$. The inclusion may be strict, and it is a non-trivial consequence of Theorem~\ref{thm:adlvDimension} that the two sets have the same maxima.
\item If $\Phi$ is irreducible, the regularity constant $C$ is equal to the \emph{Coxeter number} of $\Phi$ and explicitly given as follows:

\begin{tabular}{lccccccccc}
Cartan Type & $A_n$ & $B_n$ & $C_n$ & $D_n$ & $E_6$ & $E_7$ & $E_8$ & $F_4$ & $G_2$ \\
$C=$&$n+1$, & $2n$, & $2n$, & $2n-2$, & $12$, & $18$, & $30$, & $12$, & $6$.
\end{tabular}
\item Unlike in Proposition~\ref{prop:ADLVSuperregular}, we get no information on the number of top dimensional irreducible components. The main advantage of Theorem~\ref{thm:adlvDimension} over Proposition~\ref{prop:ADLVSuperregular} comes from the different regularity conditions, making Theorem~\ref{thm:adlvDimension} more applicable.
\item The unique minimum in $\wts(v\Rightarrow\sigma(wv))$ from \crossRef{Proposition }{prop:qbgVsDbg} corresponds to the unique maximum in $B(G)_x$. This recovers the formula for the generic Newton point from \cite[Proposition~3.1]{He2024_demazure} in the setting of Theorem~\ref{thm:adlvDimension}.
\item If the difference between $v^{-1}\mu$ and $\nu(b)$ becomes sufficiently large, the maximum $\max(E)$ can be expected to be $\ell(v^{-1}\sigma(wv))$ (cf.\ \crossRef{Lemma~}{lem:qbgNonEmptiness} or Corollary \ref{cor:wtsSpecialEstimates} (b) above) and we recover the notion of virtual dimension from He \cite[Section~10]{He2014}. In fact, one may use Corollary \ref{cor:wtsSpecialEstimates} (b) to recover \cite[Theorem~1.1]{He2021a} in the situation of Theorem~\ref{thm:adlvDimension}. This line of argumentation is ultimately cyclic, since a special case of \cite[Theorem~1.1]{He2021a} was used in the proof of Corollary \ref{cor:wtsSpecialEstimates} (b). We may however summarize that Corollary \ref{cor:wtsSpecialEstimates} (b) is the double Bruhat theoretic correspondent of \cite[Theorem~1.1]{He2021a}. Similarly, most known results on affine Deligne--Lusztig varieties correspond to theorems on the double Bruhat graph and vice versa.
\item The proof method for Theorem~\ref{thm:adlvDimension} is similar to the proof of \cite[Proposition~11.5]{He2014} or equivalently the proof of \cite[Theorem~1.1]{He2021a}.
\end{enumerate}
\end{remark}
\begin{proof}[Proof of Theorem~\ref{thm:adlvDimension}]
We assume without loss of generality that the group $G$ is of adjoint type, following \cite[Section~2]{Goertz2015}. This allows us to find a coweight $\mu_v\in X_\ast(T)_{\Gamma_0}$ satisfying for each simple root $\alpha\in\Delta$ the condition
\begin{align*}
\langle \mu_v,\alpha\rangle = \Phi^+(-v \alpha) = \begin{cases}1,&v \alpha\in \Phi^-,\\
0,&v \alpha\in \Phi^+.\end{cases}
\end{align*}
It follows that $\langle \mu_v,\beta\rangle\geq \Phi^+(-v \beta)$ for all $\beta\in \Phi^+$. Define
\begin{align*}
x_1 := wv \varepsilon^{v^{-1}\mu - \mu_v},\qquad x_2 = v^{-1}\varepsilon^{v \mu_v}\in\widetilde W.
\end{align*}
By choice of $\mu_v$, we see that $v^{-1}\mu-\mu_v$ is dominant and $(C-1)$-regular. The above estimate $\langle \mu_v,\beta\rangle \geq \Phi^+(-v \beta)$ implies $v\in \LP(x_2)$. Hence $x = x_1 x_2$ is a length additive product. We obtain
\begin{align*}
T_x = T_{x_1} T_{x_2}\equiv T_{\sigma^{-1}(x_2)} T_{x_1}\pmod{[\mathcal H,\mathcal H]_\sigma}.
\end{align*}
Define the multiset $Y$ via
\begin{align}
T_{\sigma^{-1}(x_2)} T_{x_1} = \sum_{(y,e)\in Y}Q^e T_{yx_1}\in\mathcal H.\label{eq:pfAdlv1}
\end{align}
Then each $(y,e)\in Y$ satisfies $y\leq \sigma^{-1}(x_2)$ in the Bruhat order. Writing $y = w_y\varepsilon^{\mu_y}$, we get $\mu_y^{\dom}\leq \sigma^{-1}(\mu_v)$ in $X_\ast(T)_{\Gamma_0}$. We estimate
\begin{align*}
\max_{\beta\in\Phi}\abs{\langle \mu_y,\beta\rangle} = \max_{\beta\in\Phi^+}\langle \mu_y^{\dom},\beta\rangle = \max_i \langle \mu_y^{\dom},\theta_i\rangle
\leq\max_i \langle \mu_v,\theta_i\rangle \leq C-1,
\end{align*}
by choice of $C$. It follows that
\begin{align*}
yx_1 = w_y wv\varepsilon^{v^{-1}\mu - \mu_v + (wv)^{-1}\mu_y}
\end{align*}
with $v^{-1}\mu - \mu_v + (wv)^{-1}\mu_y$ being dominant. For any dominant coweight $\lambda\in X_\ast(T)_{\Gamma_0}$, we can multiply \eqref{eq:pfAdlv1} by $T_{\varepsilon^\lambda}$ to obtain
\begin{align*}
T_{\sigma^{-1}(x_2)} T_{x_1\varepsilon^\lambda} =T_{\sigma^{-1}(x_2)} T_{x_1}T_{\varepsilon^\lambda} = \sum_{(y,e)\in Y} T_{yx_1} T_{\varepsilon^\lambda} = \sum_{(y,e)\in Y} T_{yx_1 \varepsilon^\lambda}.
\end{align*}
In light of Lemma~\ref{lem:YsetFacts}, we see that the multiset $Y$ is equal to the multiset $Y(\sigma^{-1}(x_2), wv)$ defined earlier.

For each $(y,e)\in Y$, write $yx_1 = \tilde w_y\varepsilon^{\tilde \mu_y}$ to define the sets
\begin{align*}
E_1(yx_1) :=& \{e\mid (\omega,e)\in \wts(1\Rightarrow\sigma(\tilde w_y))\text{ s.th.\ }\lambda(b) = \tilde \mu_y-\omega\in X_\ast(T)_\Gamma\},\\
E_2(yx_1) :=& \{e\mid (\omega,e)\in \wts(1\Rightarrow\sigma(\tilde w_y))\text{ s.th.\ }\nu(b) = \conv(\tilde \mu_y-\omega)\}.
\end{align*}
Define $E(yx_1)$ to be $E_1(yx_1)$ or $E_2(yx_1)$ depending on whether $E$ was chosen as $E_1$ or $E_2$. By Lemma \ref{lem:YsetIdentities} (a), we may write $\wts(\sigma^{-1}(v)\Rightarrow wv)$ as the additive union of multisets
\begin{align}
&\wts(\sigma^{-1}(v)\Rightarrow wv)\notag
\\=
&\bigcup_{(w_y\varepsilon^{\mu_y},e)\in Y(\sigma^{-1}(x_2), wv)} \{(\mu_v - (wv)^{-1} \mu_y + \omega,e+\ell)\mid (\omega,\ell)\in \wts(1\Rightarrow w_y wv)\}_m.\notag
\\=&\bigcup_{(y,e)\in Y} \{(v^{-1}\mu-\tilde \mu_y + \omega,e+\ell)\mid (\omega,\ell)\in \wts(1\Rightarrow \tilde w_y)\}_m.\label{eq:pfAdlv2}
\end{align}
Note that the definition of the sets $E_1,E_2, E_1(yx_1), E_2(yx_1)$ does not change if we apply $\sigma^{-1}$ to the occurring weights $\omega$. Hence \eqref{eq:pfAdlv2} implies
\begin{align*}
E = \bigcup_{(y,e)\in Y} \{e+\ell\mid \ell\in E(yx_1)\}.
\end{align*}

By definition of the multiset $Y$, the class polynomials of $f_{x,\mathcal O}$ for arbitrary $\sigma$-conjugacy classes $\mathcal O\subset \widetilde W$ are given by
\begin{align*}
f_{x,\mathcal O} = \sum_{(y,e)\in Y} Q^e f_{yx_1,\mathcal O}.
\end{align*}
By Theorem~\ref{thm:adlvViaClassPolynomials}, we see that $X_x(b)\neq\emptyset$ if and only if $X_{yx_1}(b)\neq\emptyset$ for some $(y,e)\in Y$. In this case, the dimension of $X_x(b)$ is the maximum of
\begin{align*}
\dim X_{yx_1}(b) + \frac 12\left(\ell(x) - \ell(yx_1)+e\right),
\end{align*}
where $(y,e)$ runs through all elements of $Y$ satisfying $X_{yx_1}(b)\neq\emptyset$.

We see that it suffices to prove the following claim for all $(y,e)\in Y$:\\

~~\begin{minipage}{.8\textwidth}$X_{yx_1}(b)\neq\emptyset$ if and only if $E(yx_1)\neq\emptyset$ and in this case, we have\\[-2em]
\begin{align*}
\dim X_{yx_1}(b) = \frac 12\left(\ell(yx_1) + \max(E(yx_1)) - \langle \nu(b),2\rho\rangle-\defect(b)\right).
\end{align*}
\end{minipage}$(\ast)$

Writing $yx_1=\tilde w\varepsilon^{\tilde \mu}$, we saw above that $\tilde \mu$ is dominant. Applying \cite[Theorem~1.2]{Milicevic2020} to the inverse of $yx_1$, or equivalently \cite[Theorem~4.2]{He2021a} directly to $yx_1$, we see that the element $yx_1$ is \emph{cordial} in the sense of \cite{Milicevic2020}. This gives a convenient criterion to check $X_{yx_1}(b)\neq\emptyset$ and to calculate its dimension. We saw in Corollary \ref{cor:wtsSpecialEstimates} (a) that the multiset $\wts(1\Rightarrow \sigma(\tilde w_y))$ must satisfy the analogous conditions. Let us recall these results.

The uniquely determined largest Newton point in $B(G)_{yx_1} = B(G)_{\tilde w\varepsilon^{\tilde \mu}}$ is $\avg_\sigma(\tilde \mu)$, cf.\ \cite[Theorem~4.2]{He2021a}.

Let $J' = \supp(\tilde w)\subseteq \Delta$ be the support of $\tilde w$ and $J = \bigcup_i \sigma^i(J') = \supp_\sigma(\tilde w)$ its $\sigma$-support. Let $\pi_J : X_\ast(T)_{\Gamma_0}\rightarrow X_\ast(T)_{\Gamma_0}\otimes\mathbb Q$ be the corresponding function from \cite[Definition~3.2]{Chai2000} or equivalently \cite[Section~3.1]{Schremmer2022_newton}. Then $\pi_J(\tilde \mu)$ is the unique smallest Newton point occurring in $B(G)_{yx_1}$, cf.\ \cite[Theorem~1.1]{Viehmann2021}.

The condition of cordiality \cite[Theorem~1.1]{Milicevic2020} implies that $B(G)_{yx_1}$ contains all those $[b]\in B(G)$ with the correct Kottwitz point $\kappa(b) = \kappa(yx_1) = \kappa(x)$ and Newton point
\begin{align*}
\pi_J(\tilde\mu)\leq \nu(b)\leq \avg_\sigma(\tilde \mu).
\end{align*}
In this case, we know moreover from \cite[Theorem~1.1]{Milicevic2020} that $X_{yx_1}(b)$ is equidimensional of dimension
\begin{align*}
\dim X_{yx_1}(b) = \frac 12\left(\ell(yx_1)+\ell(\tilde w)-\langle \nu(b),2\rho\rangle-\defect(b)\right).
\end{align*}
This condition on Newton points is equivalent to $\avg_\sigma(\tilde \mu)-\nu(b)$ being a non-negative $\mathbb Q$-linear combination of simple coroots of $J$, or equivalently $\tilde\mu-\lambda(b)$ being a non-negative $\mathbb Z$-linear combination of these coroots.

On the double Bruhat side, note that $(\omega,e)\in \wts(1\Rightarrow\tilde w)$ implies $\omega\in \mathbb Z\Phi_{J'}^\vee$ and $e\leq \ell(\tilde w)$. This can either be seen directly, similar to the proof of \crossRef{Lemma~}{lem:qbgNonEmptiness}, or as in Corollary~\ref{cor:wtsSpecialEstimates}, reducing to \cite[Theorem~1.1]{Viehmann2021}. From Corollary~\ref{cor:wtsSpecialEstimates}, we know conversely that any $\omega\geq 0$ with $\omega\in \mathbb Z\Phi^\vee_{J'}$ satisfies $(\omega,\ell(\tilde w))\in \wts(1\Rightarrow\tilde w)$.

Comparing these explicit descriptions of $\dim X_{yx_1}(b)$ and $\max(E(yx_1))$, we conclude the claim $(\ast)$. This finishes the proof.
\end{proof}
\section{Outlook}\label{sec:outlook}
We saw that the weight multiset of the double Bruhat graph can be used to describe the geometry of affine Deligne--Lusztig varieties in many cases. This includes the case of superparabolic elements $x$ together with sufficiently large integral $[b]\in B(G)$ in split groups \crossRef{Theorem~}{thm:adlvViaSemiInfiniteOrbits}, as well as the case of sufficiently regular elements $x$ together with arbitrary $[b]\in B(G)$ (Theorem~\ref{thm:adlvDimension}). One may ask how much the involved regularity constants can be improved, and whether a unified theorem simultaneously generalizing \crossRef{Theorem~}{thm:adlvViaSemiInfiniteOrbits} and Theorem~\ref{thm:adlvDimension} can be found. Towards this end, we propose a number of conjectures that would generalize our theorems in a straightforward manner.

Let $x=w\varepsilon^\mu\in\widetilde W$ and $[b]\in B(G)$. If $X_x(b)\neq\emptyset$, define the integer $D\in\mathbb Z_{\geq 0}$ such that
\begin{align*}
\dim X_x(b) = \frac 12\left(\ell(x)+D-\langle \nu(b),2\rho\rangle-\defect(b)\right),
\end{align*}
and denote the number of $J_b(F)$-orbits of top dimensional irreducible components in $X_x(b)$ to be $C\in\mathbb Z_{\geq 1}$. We would like to state the following conjectures. The first conjecture makes a full prediction of the nonemptiness pattern and the dimension for elements $x$ in the shrunken Weyl chamber and arbitrary $[b]\in B(G)$.
\begin{conjecture}\label{conj:shrunkenX}
Suppose that $x$ lies in a shrunken Weyl chamber, i.e.\ $\LP(x) = \{v\}$ for a uniquely determined $v\in W$. Define $E$ to be either of the multisets
\begin{align*}
E_1 := &\{e\mid (\omega,e)\in\wts(v\Rightarrow\sigma(wv))\text{ s.th.\ }\lambda(b) \equiv v^{-1}\mu-\omega\in X_\ast(T)_{\Gamma}\}_m,\\
E_2 := &\{e\mid (\omega,e)\in\wts(v\Rightarrow\sigma(wv))\text{ s.th.\ }\nu(b) =\conv(v^{-1}\mu-\omega)\}_m.
\end{align*}
We make the following predictions.
\begin{enumerate}[(a)]
\item $X_x(b)\neq\emptyset$ if and only if $E\neq\emptyset$ and $\kappa(x) = \kappa(b)\in\pi_1(G)_{\Gamma}$ (the latter condition on Kottwitz points is automatically satisfied if $E=E_1$).
\item If $X_x(b)\neq\emptyset$, then $\max(E) = D$.
\item If $X_x(b)\neq\emptyset$, then $C$ is at most the multiplicity of $D$ in $E$ (which may be $+\infty$ for $E_2$).
\end{enumerate}
\end{conjecture}
The multiset $E_1$ is always contained in $E_2$, since $\nu(b) = \conv(\lambda(b))$. The inclusion may be strict. So in fact we are suggesting two different dimension formulas for shrunken $x$, and claim that both yield the same answer, which moreover agrees with the dimension.

For sufficiently regular $x$, Theorem~\ref{thm:adlvDimension} shows (a) and (b). Under some strong superregularity conditions, Proposition~\ref{prop:ADLVSuperregular} shows (c) with equality. While both proofs can certainly be optimized with regards to the involved regularity constants, proving Conjecture~\ref{conj:shrunkenX} as stated will likely require further methods. It is unclear how to show the conjecture e.g.\ for the particular element $x = w_0\varepsilon^{-2\rho^\vee}$, since the proof method for Theorem~\ref{thm:adlvDimension} fails.

It is easy to see that Conjecture~\ref{conj:shrunkenX} is compatible with many known results on affine Deligne--Lusztig varieties, such as the ones recalled in the introduction of the previous article \crossRef{Theorem~}{thm:virtDim}. By Corollary~\ref{cor:wtsSpecialEstimates}, we see that parts (a) and (b) of Conjecture~\ref{conj:shrunkenX} hold true for cordial elements $x$. If $x$ is of the special form $x = w_0\varepsilon^\mu$ with $\mu$ dominant, then $x$ is in a shrunken Weyl chamber and we know that (c) holds, cf.\ \crossRef{Remark~}{rem:chenZhu}.

Our second conjecture suggests how the double Bruhat graph can be used for elements $x$ which are not necessarily in shrunken Weyl chambers.
\begin{conjecture}\label{conj:integralB}
Suppose that $[b]$ is \emph{integral}, i.e.\ of defect zero.

Define for each $v\in \LP(x)$ and $u\in W$ the multiset
\begin{align*}
E(u,v) := \{e \mid (\omega,e)\in \wts(u\Rightarrow \sigma(wu)\dashrightarrow \sigma(wv))\text{ s.th.\ }u^{-1} \mu-\omega =\lambda(b) \in X_\ast(T)_{\Gamma}\}_m.
\end{align*}
Set $\max\emptyset :=-\infty$ and define
\begin{align*}
d :=& \max_{u\in W} \min_{v\in \LP(x)} \max(E(u,v))\in\mathbb Z_{\geq 0}\cup \{-\infty\}.\\
c :=& \sum_{u\in W} \min_{v\in \LP(x)}\left(\text{multiplicity of $d$ in $E(u,v)$}\right)\in\mathbb Z_{\geq 0}.
\end{align*}
We make the following predictions.
\begin{enumerate}[(a)]
\item If there exists for every $u\in W$ some $v\in \LP(x)$ with $E(u,v)=\emptyset$, i.e.\ if $d=-\infty$, then $X_x(b)=\emptyset$.
\item If $X_x(b)\neq\emptyset$, then $D\leq d$.
\item If $X_x(b)\neq\emptyset$ and $D=d$, then $C\leq c$.
\end{enumerate}
\end{conjecture}
If the group is split, then \crossRef{Theorem~}{thm:adlvViaSemiInfiniteOrbits} proves (a), (b) and (c). Moreover, under some strong superparabolicity assumptions, we get the full conjecture including equality results for (b) and sometimes (c). We expect that a similar superparabolicity statement holds true for non-split groups, but it is unclear what the involved regularity constants should be, which is why we did not formulate a precise, falsifiable conjecture.

If the element $x\in\widetilde W$ is in a shrunken Weyl chamber with $\LP(x) = \{v\}$, then the multiset $E_1$ from Conjecture~\ref{conj:shrunkenX} is equal to the multiset $E(v,v)$ from Conjecture~\ref{conj:integralB}. If we moreover assume that Conjecture~\ref{conj:shrunkenX} holds true, then we get parts (a), (b) and (c) of Conjecture~\ref{conj:integralB}.

Compatibility of Conjecture~\ref{conj:integralB} with previously known results is a lot harder to verify. We expect that one does not have to account for all pairs $(u,v)$ as in Conjecture~\ref{conj:integralB} to accurately describe nonemptiness and dimension of $X_x(b)$, similar to \crossRef[(c)]{Theorem }{thm:adlvViaSemiInfiniteOrbits} or Conjecture~\ref{conj:shrunkenX}. However, we cannot make precise prediction how such a refinement of Conjecture~\ref{conj:integralB} should look like in general.

Nonetheless, extensive computer searches did not yield a single counterexample to either conjecture. Most straightforward generalizations of these conjectures, however, can be disproved quickly using such a computer search \cite{sagemath, sage-combinat}.
\begin{example}
For both conjectures, the estimate on the number of irreducible components is only an upper bound. Indeed, it suffices to consider elements of the form $x = w_0\varepsilon^{\mu}$ for dominant cocharacters $\mu$. Then, as discussed in \crossRef{Remark~}{rem:chenZhu}, the number $C$ is equal to the dimension of the $\lambda(b)$-weight space of the \emph{irreducible Weyl module} $M_\mu$. The element $x$ lies in a shrunken Weyl chamber, and the multiplicity of $d=D$ in $E_1 = E(v,v)$ is equal to the dimension of the $\lambda(b)$-weight space in the Verma module $V_\mu$. These numbers are not equal in general.
\end{example}
\begin{example}
One may ask whether it is possible to find for each non-shrunken $x$ an element $v\in \LP(x)$ such that the analogous statement of Conjecture~\ref{conj:shrunkenX} holds true. While this is certainly possible, say, for cordial elements $x$, such a statement cannot be expected to hold true in general. We may choose $G = \GL_4$ and $x = s_3 s_2 s_1 \varepsilon^\mu$ where the pairing of $\mu$ with the simple roots $\alpha_1,\alpha_2,\alpha_3$ is given by $1,-1,1$ respectively. Then $\LP(x) = \{s_2, s_2 s_3\}$. For $[b]$ basic, we have $D=3$, yet the analogous statements of Conjecture~\ref{conj:shrunkenX} for both possible choices of $v$ in $\LP(x)$ would predict $D=5$.
\end{example}
\begin{example}
Conjecture~\ref{conj:integralB} should not be expected to hold for non-integral $[b]$. Indeed, it suffices to choose $G=\GL_3$ and $x=w\varepsilon^\mu$ to be of length zero such that the action of $x$ on the affine Dynkin diagram is non-trivial. Let $[b] = [x]$, so that $B(G)_x = \{[b]\}$. Define
\begin{align*}
E(u,v) := \{e \mid (\omega,e)\in \wts(u\Rightarrow wu\dashrightarrow wv)\text{ s.th.\ }u^{-1} \mu-\omega =\lambda(b) \in X_\ast(T)_{\Gamma}\}_m
\end{align*}
for $u,v\in W = \LP(x)$. Since $w\neq 1$, we have $E(u,v)=\emptyset$ whenever $v = u w_0$. A statement analogous to Conjecture \ref{conj:integralB} (a) would thus predict that $X_x(b)=\emptyset$, which is absurd.
\end{example}
\begin{example}
The converse of Conjecture \ref{conj:integralB} (a) should not be expected to hold, even for $[b]$ basic. The construction in Conjecture~\ref{conj:integralB} can fail to detect $(J,w,\delta)$-alcove elements, hence falsely predict a non-empty basic locus. For a concrete example, one may choose $G = \GL_3$ and $x$ to be the shrunken element $x = s_2\varepsilon^{\rho^\vee}$, with $\langle\rho^\vee,\alpha\rangle=1$ for all simple roots $\alpha$. Then $\LP(x) = \{1\}$. For $u = s_1 s_2$ and $[b] = [1]$ basic, we have $E(u,1)\neq\emptyset$.
\end{example}

\begin{CJK*}{UTF8}{gbsn}
\printbibliography
\end{CJK*}

\end{document}